\documentclass[11pt, psamsfonts, reqno]{amsart}

\usepackage[T1]{fontenc}
\usepackage[latin1]{inputenc}
\usepackage[french, english]{babel}
\usepackage{amsthm}
\usepackage{amscd}
\usepackage{amssymb,amsmath,bbm}
\usepackage[all]{xy}
\usepackage{mathrsfs}
\usepackage[left=0.0cm,right=0.0cm,top=2.7cm,bottom=3cm]{geometry}
\usepackage{hyperref}
\usepackage{verbatim}

\theoremstyle{plain}
\newtheorem{theorem}{Theorem}[section]

\newtheorem{lemma}[theorem]{Lemma}
\newtheorem{proposition}[theorem]{Proposition}

\newtheorem*{theorem*}{Theorem}

\theoremstyle{definition}
\newtheorem{definition}[theorem]{Definition}
\newtheorem{example}[theorem]{Example}

\theoremstyle{remark}
\newtheorem{remark}[theorem]{Remark}

\usepackage[usenames,dvipsnames]{color}

\newcommand{\mycomment}[1]{%
}%

\usepackage{tikz}
\usetikzlibrary{calc}

\def\Q{{\bf Q}}

\def\Z{{\bf Z}}
\def\C{{\bf C}}
\def\N{{\bf N}}
\def\R{{\bf R}}

\def\O{{\mathcal{O}}}
\def\H{{H}}

\def\Qbar{{\overline{{\bf Q}}}}

\def\zp{{\Z_p}}

\def\qp{{\Q_p}}

\def\Hom{\mathrm{Hom}}

\def\A{\mathbf{A}}
\def\Af{{\bf A}_{f}}

\def\j{j}
\def\Sc{\mathcal{S}}
\def\epsilon{\varepsilon}

\def\Sh{{\operatorname{Sh}}}

\def\GSp{{\mathrm{GSp}}}

\def\Sp{{\mathrm{Sp}}}
\def\SL{{\mathrm{SL}}}
\def\GL{\mathrm{GL}}
\def\Gm{\mathrm{G}_{\rm m}}
\def\G{\mathbf{G}}
\def\H{\mathbf{H}}

\def\B{\mathbf{B}}
\def\T{\mathbf{T}}

\def\matrix#1#2#3#4{{\big(\begin{smallmatrix}#1&#2\\ #3&#4\end{smallmatrix}\big)}}

\title[On higher regulators of Siegel varieties]{On higher regulators of Siegel varieties}

\author{Antonio Cauchi{,} Francesco Lemma and Joaqu\'in Rodrigues Jacinto}

\makeatletter
\def\@tocline#1#2#3#4#5#6#7{\relax
  \ifnum #1>\c@tocdepth 
  \else
    \par \addpenalty\@secpenalty\addvspace{#2}%
    \begingroup \hyphenpenalty\@M
    \@ifempty{#4}{%
      \@tempdima\csname r@tocindent\number#1\endcsname\relax
    }{%
      \@tempdima#4\relax
    }%
    \parindent\z@ \leftskip#3\relax \advance\leftskip\@tempdima\relax
    \rightskip\@pnumwidth plus4em \parfillskip-\@pnumwidth
    #5\leavevmode\hskip-\@tempdima
      \ifcase #1
       \or\or \hskip 1em \or \hskip 2em \else \hskip 3em \fi%
      #6\nobreak\relax
    \dotfill\hbox to\@pnumwidth{\@tocpagenum{#7}}\par
    \nobreak
    \endgroup
  \fi}
\makeatother
\usepackage{hyperref}

\begin{document}
\maketitle
\selectlanguage{english}

\begin{abstract}
Inspired by Beilinson conjectures, we construct classes in the middle degree plus one motivic cohomology of the Siegel Shimura variety of almost any dimension. We compute their image by Beilinson's higher regulator in terms of Rankin-Selberg type automorphic integrals. Our construction generalises the one for $\GSp(4)$ in \cite{lemmarf} and for $\GSp(6)$ in \cite{CLRGSp6}. For Siegel varieties associated to small genus symplectic groups, we also show how these integrals unfold.
\end{abstract}

\setcounter{tocdepth}{2}
\tableofcontents

\section{Introduction}

The conjectures of Beilinson and Deligne relate special values of $L$-functions to the motivic cohomology of an algebraic variety defined over a function field. They widely generalize  the class number formula for the Hasse-Weil zeta function of a number field and they are related to the celebrated Birch and Swinnerton-Dyer conjecture for the $L$-function of an elliptic curve. \\

One of the first steps in the way to proving instances of Beilinson conjectures is to construct potentially interesting classes in the motivic cohomology of an algebraic variety and to compute their image under the so called archimedean regulator into Deligne cohomology. These are often related to certain adelic integrals which, in the few known instances of the conjecture, are related to the special values of the pertaining $L$-function. \\

In \cite{lemmarf}, the second named author constructed classes in the fourth dimensional motivic cohomology group of the Siegel threefold and related their archimedean regulators to non-critical values of the Spin $L$-function of automorphic forms of $\GSp_4$. In \cite{CLRGSp6}, following \cite{GSp6paper1} and based on the work \cite{PollackShah} of Pollack and Shah, we constructed classes in the seventh dimensional motivic cohomology group of the Siegel sixfold and also related their archimedean regulators to non-critical values of the Spin $L$-function of automorphic forms of $\GSp_6$. The main purpose of this article is to construct classes in the middle degree plus one motivic cohomology of Siegel varieties and to prove a formula relating their image under Beilinson regulator to certain adelic integrals of Rankin-Selberg type. For low rank cases, we show that these integrals unfold to certain Fourier coefficients. \\

Let $\G = \GSp_{2n}$ be the symplectic similitude group of rank $n$ and denote by $\Sh_{\G}$ the Shimura variety associated to $\G$. These Shimura varieties are moduli spaces for principally polarized abelian schemes with certain level structure and their cohomology plays a key role in the arithmetic theory of automorphic forms, particularly in the Langlands program. For instance, they are known to realise \cite{KretShin} the Galois representation associated to certain cuspidal automorphic representations of $\G(\A)$. Our first task is to propose a construction, for almost all values of $n \geq 2$, of certain cohomology classes
\[ \mathrm{Eis}_{\mathcal{M}, n} \in H^{d + 1}_\mathcal{M}(\Sh_{\G}, \overline{\Q}(t)), \]
where $d = \frac{n (n+1)}{2}$ denotes the dimension of $\Sh_{\G}$ and $t = \tfrac{1}{2}(d+\eta+2)\in \Z$, where we let $\eta = 0$ if $d$ is even and $\eta  = 1$ if $d$ is odd. Our expectations and interest on these motivic classes are supported by the fact that Beilinson's conjectures predict that the $\pi_f$-isotypic component of $H^{d + 1}_\mathcal{M}(\Sh_{\G}, \overline{\Q}(t))$ is non-zero  whenever the Spin $L$-function of a cuspidal automorphic representation $\pi= \pi_\infty \otimes \pi_f$ of $\G(\A)$ vanishes at $\tfrac{1}{2}$ to the left of the center of symmetry with respect to the functional equation. This latter condition is often satisfied: for instance, each cohomological cuspidal automorphic representation which is stable at infinity does.  \\

Our construction is based on the existence of certain special algebraic cycles inside $\Sh_\G$ and cohomology classes on these cycles coming from modular units. We show, using an elementary combinatorial lemma, that there exists a subgroup $\H$ of $\G$ satisfying certain particular properties. The subgroup $\H$ is constructed as a fiber product over $\GL_1$ of smaller rank symplectic groups over some totally real fields and it has a natural projection to $\GL_2$ if $\eta = 0$ and to $\GL_2 \times_{\GL_1} \GL_2$ if $\eta = 1$. Some examples are
\begin{align*}
& \GSp_2 \boxtimes \GSp_2 \hookrightarrow \GSp_4, \;\;\;\;\;\; & &\GSp_2 \boxtimes \GSp_{2, F}^* \hookrightarrow \GSp_6, \\
&\GSp_2 \boxtimes \GSp_2 \boxtimes \GSp_4 \hookrightarrow \GSp_8, \;\;\;\;\;\;
& &\GSp_2 \boxtimes \GSp_2 \boxtimes \GSp_6 \hookrightarrow \GSp_{10}, \\
&\GSp_2 \boxtimes \GSp_4 \boxtimes \GSp_8 \hookrightarrow \GSp_{14}.
\end{align*}
In the second example, $F$ is any quadratic \'etale algebra over $\Q$, i.e. $F$ is either $\Q \times \Q$ or a quadratic totally real extension of $\Q$. Denoting by $\iota: \H \to \G$ the embedding of groups, we show that $\iota$ induces a closed embedding of Shimura varieties $\iota: \Sh_\H \to \Sh_\G$ of codimension $(d - \eta)/2$. The second basic input for our construction are modular units. These are elements of the motivic cohomology groups $H^1_\mathcal{M}(\Sh_{\GL_2}, \overline{\Q}(1)) \cong \mathcal{O}(\Sh_{\GL_2})^\times \otimes_\Z \overline{\Q}$, which can be seen as motivic realisations of Eisenstein series. Indeed, by the second Kronecker limit formula, their logarithm is related to limiting values of some weight 0 real analytic Eisenstein series. The elements $\mathrm{Eis}_{\mathcal{M}, n} \in H^{d + 1}_\mathcal{M}(\Sh_{\G}, \overline{\Q}(t))$ are then constructed as the push-forward to $\Sh_\G$ of the pullback to $\Sh_\H$ of one or two modular units, i.e. letting $\mathrm{pr} : \Sh_\H \to \Sh_{\H'}$, where $\H' = \GL_2$ or $\GL_2 \otimes_{\GL_1} \GL_2$ according to whether $\eta = 0$ or $1$, and letting $u$ be either one modular unit or a cup product of two modular units, we define
\[ \mathrm{Eis}_{\mathcal{M}, n} = \iota_* \mathrm{pr}^*(u). \]

Our next task is to study the archimedean regulators of these classes. Recall the existence of Beilinson's regulator map
\[  r_\mathcal{D}: H^{{d}+1}_{\mathcal{M}}(\Sh_{\G}(U), \overline{\Q}(t)) \to  H^{{d} + 1}_{\mathcal{D}}(\Sh_{\G}(U)/\R, \R(t)) \otimes_{\R} \overline{\Q}. \]
According to Beilinson-Deligne conjectures, if the latter Deligne--Beilinson cohomology group is non-zero, one expects to be able to construct non-zero motivic cohomology elements which are related to special values of $L$-functions.
One basic idea to deal with these regulators, which comes from Beilinson, is to use Poincar\'e duality to pair the images of these classes with a suitably constructed differential form associated to a cuspidal form in an automorphic representation of $\G(\A)$. Let $\pi = \pi_\infty \otimes \pi_f$ be a cuspidal automorphic representation of $\G(\A)$  with trivial central character, for which $\pi_\infty$ is a discrete series  and such that $\pi_f$ has a non-zero vector fixed by $U$. Associated to a cusp form $\Psi = \Psi_\infty \otimes \Psi_f \in \pi_\infty \otimes \pi_f^U$ such that $\Psi_\infty$ is a highest weight vector of one minimal $K_\infty$-type of $\pi_\infty$, there is a harmonic differential form $\omega_\Psi$ on $\Sh_{\G}(U)$. Our first main result is the following.

\begin{theorem} \label{IntroTheoA}
The differential form $\omega_\Psi$ induces a natural linear form
\[ \langle \,\, , \omega_\Psi \rangle : H^{d+1}_{\mathcal{D}}(\Sh_{\G}(U)/\R, \R(t)) \otimes_{\R} \overline{\Q} \rightarrow \C \otimes_{\Q} \Qbar{} \]
and we have
\[ \langle r_\mathcal{D}({\rm Eis}_{\mathcal{M}, n}), \omega_\Psi \rangle  = \int_{\Sh_\H(V)} \xi \wedge \iota_n^*\omega_\Psi, \]
where $\xi = \mathrm{pr}^* \log |u|$ if $\eta = 0$ and a differential form constructed from two modular units if $\eta = 1$.
\end{theorem}

The proof of this results relies on the methods developed in \cite{CLRGSp6}, where we gave a description of Deligne--Beilinson cohomology in terms of complexes of \textit{tempered currents}, i.e. sheaves of continuous linear forms of rapidly decreasing differential forms. The use of usual currents, i.e. continuous linear forms of differential forms defined on the compactification, to calculate Deligne--Beilinson cohomology has already been considered before, cf. \cite{Jannsen} or \cite{BKK}, but this description does not seem to suffice for our purposes. \\

We now rephrase the above result in terms of a Rankin-Selberg integral, using the second Kronecker limit formula. Let $K_{\H}$ denote a maximal compact subgroup of $\H(\R)$ and fix a generator $X_0$ of the highest exterior power of $({\rm Lie} \, \H(\R)/{\rm Lie} \, K_{\H} ) \otimes_{\R} \C$. Moreover, denote by $E(h, s)$ the weight 0 real analytic Eisenstein series on $\GL_2$ appearing in the second Kronecker limit formula (cf. Proposition \ref{KLF}). 

\begin{theorem} \label{IntroTheoA'}
Let $n$ be congruent to $0$ or $3$ mod $4$. We have
\begin{equation*}
\langle r_\mathcal{D}({\rm Eis}_{\mathcal{M}, n}), \omega_\Psi \rangle =  C_U \int_{\H(\Q) \Z_\G(\A) \backslash \H(\A)} E(h_1,0) (A.\Psi)(h) dh,
\end{equation*}
for a certain $A \in \mathcal{U}(\mathfrak{g}_\C)$ (defined precisely in \S \ref{adelicintegralsec}), where $C_{U}$ is a volume factor depending on $U \cap \H$.
\end{theorem}

We also give a formula (Theorem \ref{adelicintegral2}) for the pairing $\langle r_\mathcal{D}({\rm Eis}_{\mathcal{M}, n}), \omega_\Psi \rangle$ in the case where $n \equiv 1,2 \pmod{4}$ as a linear combination of integrals of Rankin-Selberg type involving Eisenstein series on $\GL_2 \times_{\GL_1} \GL_2$ and the restriction of the cusp form $\Psi$ to $\H$. In particular, Theorem \ref{adelicintegral2} for $n=2$ strengthens the results of \cite{lemmarf}, where the second-named author worked in the setting where the weight of the local system is enough regular, hence excluding the case of trivial coefficients.\\

When $n = 2$ or $n = 3$, the integrals on the right hand side of the formula of Theorem \ref{IntroTheoA'} are related to non-critical special values of Spin $L$-functions, cf. \cite{lemmarf} and \cite{CLRGSp6}. For $n \geq 4$, these Rankin-Selberg integrals have not been studied in the literature. As a first step towards relating them to $L$-functions, we show that these integrals for $n=4,5,7$ unfold to certain Fourier coefficients. Fourier coefficients of automorphic forms can be parametrised by unipotent orbits and, for the symplectic group $\GSp_{2n}$, these are in correspondence with partitions $(n_1 \, n_2 \hdots n_k)$ of $2n$ with odd numbers appearing with even multiplicity \cite{GinzburgRallisSoudry}. To such a partition one can associate a Fourier coefficient (cf. Definition \ref{fouriercoeff} and Definition \ref{fouriercoeffwithextraint}). 

\begin{theorem}[Proposition \ref{unfoldingforgsp8}, Proposition \ref{unfoldingforgsp10}, Proposition \ref{unfoldingforgsp14}]
If $n = 4, 5, 7$, the integral of Theorem \ref{IntroTheoA'} unfolds to an adelic integral over a Fourier coefficient of type $(4\, 2\, 1^2)$, $(2^2\, 1^6)$ and $(4 \, 2 \, 1^8)$, respectively.
\end{theorem}

\textbf{Acknowledgements.} We would like to thank David Ginzburg for kindly explaining to us many unfolding techniques. We would like to thank Aaron Pollack for several fruitful discussions. The first named author is supported by the NSERC grant RGPIN-2018-04392 and Concordia Horizon postdoc fellowship n.8009. The first named author's research has also been supported by the European Research Council (ERC) under the European Union's Horizon 2020 research and innovation programme (grant agreement No. 682152). The third named author has received financial support from ERC-2018-COG-818856-HiCoShiVa.

\section{Preliminaries}

In this section, we introduce some notation and state a lemma concerning the existence of certain subgroups of the symplectic groups. We use them to define the motivic cohomology classes we will study later.

\subsection{Groups} \label{section-groups}

Let $\GSp_{2n}$ be the group scheme over $\Z$ whose $R$-points, for any commutative ring $R$ with identity, are described by
\[ \GSp_{2n}(R) = \{ A \in \GL_{2n}(R) \; : \; {}^t A J A = \nu(A) J, \; \nu(A) \in \Gm(R) \}, \] where $J$ is the matrix ${ \matrix 0 {J_n} {-J_n} 0}$, for $J_n$ denoting the $n\times n$ anti-diagonal matrix with diagonal entries equal to 1.

\subsubsection{Subgroups}

Let $F$ be a totally real $\Q$-algebra of dimension $\delta$;
denote by $\GSp_{2m,F}^* /\Q$ the subgroup scheme of ${\rm Res}_{F/\Q} \GSp_{2m,F}$ sitting in the Cartesian diagram 
\[ \xymatrix{ 
\GSp_{2m,F}^* \ar@{^{(}->}[r] \ar[d] & {\rm Res}_{F/\Q} \GSp_{2m,F} \ar[d]^{\nu} \\ 
\mathbf{G}_m  \ar@{^{(}->}[r] & {\rm Res}_{F/\Q} \mathbf{G}_{{\rm m},F}.
}
\]

For instance, when $F=\Q^{\delta}$,
\[\GSp_{2m,F}^*=\GSp_{2m}^{ \boxtimes \delta} =\{ (g_i) \in \GSp_{2m}^\delta \; :\; \nu(g_1) = \cdots = \nu(g_\delta) \}.\]

Consider $F^{2m}$ with its standard $F$-alternating form $\langle \; ,\; \rangle_F$. We fix the standard symplectic $F$-basis $\{ e_1, \hdots , e_m , f_m,  \hdots , f_1  \}$ and define $\langle \; ,\; \rangle_\Q$ to be ${\rm Tr}_{F/\Q} \circ \langle \; ,\; \rangle_F$. Then, by definition $\GSp_{2m,F}^* \subset \GSp(\langle \; ,\; \rangle_\Q)$. Notice that, after opportunely fixing a $\Q$-basis of $F$, $\GSp(\langle \; ,\; \rangle_\Q)$ becomes isomorphic to $\GSp_{2m\delta}$, thus we have an embedding \begin{eqnarray}\label{subemb} \GSp_{2m,F}^* \hookrightarrow \GSp(\langle \; ,\; \rangle_\Q) \simeq \GSp_{2m\delta}. \end{eqnarray}

\begin{example} 
Let $F$ be a real \'etale quadratic extension over $\Q$. Such extensions are parametrised by $a \in \Q^\times_{> 0} / ( \Q^\times _{> 0})^2,$ and we identify $F = \Q \oplus \Q \sqrt{a}$, for a representative $a$ of the corresponding class in $\Q^\times_{> 0} / ( \Q^\times_{> 0} )^2$.
Let $m=1$; we realise the isomorphism $\GSp(\langle \; ,\; \rangle_\Q) \simeq \GSp_{4}$, by choosing the $\Q$-basis of $F^2$ given by 
\[\{ \tfrac{1}{2\sqrt{a}}e_1 , \tfrac{1}{2}e_1 , f_1, \sqrt{a}f_1 \}.\] Indeed, such a basis represents the alternating form $\langle \; ,\; \rangle_\Q$ as given by $J$.
\end{example}

Let $V_n$ be the standard representation of $\GSp_{2n}$ with symplectic basis $\{ e_i, f_j  \}$. Given a partition $(n_i)_{1\leq i \leq t}$ of $n$, we will consider the embedding \begin{eqnarray}\label{block} \GSp_{2n_1} \boxtimes \cdots \boxtimes \GSp_{2n_t} \hookrightarrow \GSp_{2n} \end{eqnarray} induced by the decomposition $V_n=\oplus_{i=1}^t V_{n_i}$, where each $V_{n_i}$ is a vector space of dimension $2n_i$ endowed with symplectic basis $\{e_{s_{i-1}+1}, \hdots ,  e_{s_i}, f_{s_i},  \hdots , f_{s_{i-1}+1} \}$, with $s_k := \sum_{j=1}^{k}n_j$.

\subsubsection{Special embeddings} \label{mod_emb}

We start with the following combinatorial result. 

\begin{lemma} \label{reznick} For almost all integers $n$, there exist $n_1 \leq \hdots \leq n_{k(n)} \in \N$ such that $\sum_{i} n_i = n$ and such that the following is true:
\begin{itemize}
 \item If $n \equiv 0$ or $3$ modulo $4$, then $n_1 = 1$ and $\sum_i \binom{n_i + 1}{2} = \frac{1}{2} \binom{n + 1}{2}$.
 \item If $n \equiv 1$ or $2$ modulo $4$, then $n_1 = n_2 = 1$ and $\sum_i \binom{n_i + 1}{2} = \frac{1}{2} \big( \binom{n + 1}{2} + 1 \big)$.
\end{itemize}
\end{lemma}

\begin{proof}
Let $n \equiv 0$ or $3$ modulo $4$, then we need to find $n_i$'s such that $1 + \sum_i n_i = n$ and such that $1+\sum_i \binom{n_i + 1}{2} = \frac{1}{2} \binom{n + 1}{2}$. Using the first condition, the second one boils down to
\[ \sum_i n_i^2 = \frac{n^2 + n}{2} - \sum_i n_i -2 = \frac{(n - 1)^2 + (n - 1) - 2}{2}. \]
Analogously, if $n \equiv 1$ or $2$ modulo $4$, then we want to find $n_i$'s such that $1 + 1 + \sum_i n_i = n$ and the second condition in the lemma becomes
\[ \sum_i n_i^2 = \frac{(n - 2)^2 + 3 (n - 2)}{2}. \]
The existence of $n_i$'s satisfying these conditions follows from \cite{Reznick} for $n$ large enough.
\end{proof}

Let $n$ be an integer as in the lemma above and let $p_n=(n_i)_i$ be a partition of $n$ given by the lemma. Note that $p_n$ is not necessarily unique. We agree that $n_1 \leq n_2 \leq \ldots \leq n_{k(n)}$. We set $\epsilon_n = 1$ (resp. $2$) if $n \equiv 0$ or $3$ (resp. $1$ or $2$) modulo $4$. Let us consider the sequence of subsets of $p_n$ defined inductively by
\begin{eqnarray*}
X_1 &=& \left\{ \begin{array}{ll}
        \{n_1\} & \mbox{if } \epsilon_n=1 \\
        \{n_1, n_2\} & \mbox{if } \epsilon_n=2\\
    \end{array}
    \right.\\
X_{s} &=& \left\{n_i\in p_n\big|n_i=\min\left\{p_n-\bigcup_{j=1}^{s-1}X_{j}\right\}\right\}, \mbox{for } s \geq 2.
\end{eqnarray*}
For any $i \geq 2$, let $\delta_i = |X_i|$, $F_s$ denote a totally real \'etale $\Q$-algebra of dimension $\delta_s$ and let $m_i$ denote the common value of the elements of $X_i$. By composing the maps of \eqref{subemb} and \eqref{block}, we construct the embedding
\[ \iota_n: \H_n := \GSp_{2} \boxtimes \GSp_{2m_2, F_2}^*  \hdots \boxtimes \GSp_{2 m_t, F_{t}}^* \hookrightarrow \GSp_{2n} =: \G_n \]
if $\epsilon_n=1$ and the embedding
\[ 
\iota_n: \H_n := \GSp_{2} \boxtimes \GSp_2 \boxtimes \GSp_{2m_3, F_3}^*  \hdots \boxtimes \GSp_{2 m_t, F_{t}}^* \hookrightarrow \G_n 
\]
if $\epsilon_n=2$. 

\begin{remark}
  Some examples of embeddings $\iota_n$ for small values of $n$ are:
  \begin{itemize}
   \item $\GSp_2 \boxtimes \GSp_2 \hookrightarrow \GSp_4$,
   \item $\GSp_2 \boxtimes \GSp_{2, F}^* \hookrightarrow \GSp_6$,
   \item $\GSp_2 \boxtimes \GSp_2 \boxtimes \GSp_4 \hookrightarrow \GSp_8$,
   \item $\GSp_2 \boxtimes \GSp_2 \boxtimes \GSp_6 \hookrightarrow \GSp_{10}$,
   \item $\GSp_2 \boxtimes \GSp_4 \boxtimes \GSp_8 \hookrightarrow \GSp_{14}$.
  \end{itemize}
  The first one was the one used by \cite{lemmarf} and \cite{LSZ1} and the second one was considered in \cite{CLRGSp6} and \cite{GSp6paper1}. We also remark that there is no partition satisfying the condition of Lemma \ref{reznick} for $n = 6, 9, 10, 13, 16, 17, 26, 33$ (and these are most probably all the exceptions of Lemma \ref{reznick}).
\end{remark}

\subsection{Shimura varieties}
We keep the notations of the previous section. In particular, we let $p_n$ be a partition of $n$ as in Lemma \ref{reznick}.
Let $\mathbf{S}={\rm Res}_{\C/\R}{\Gm}_{/\C}$ be the Deligne torus. After identifying $\GSp_{2m_s,F_s}^*/\R$ with $\GSp_{2m_s/\R}^{\boxtimes \delta_s}$, denote by $X_{\H_n}$ the $\H_n(\R)$-conjugacy class of \[h:\mathbf{S} \longrightarrow {\H_n^{\epsilon_n}}_{/\R}, \quad x+iy \mapsto \left( { \matrix {x} {y } {-y} {x} }, \hdots , { \matrix {x I_{m_i}} {y J_{m_i}} {-y J_{m_i}} {x I_{m_i}}, \hdots }, { \matrix {x I_{m_t}} {y J_{m_t}} {-y J_{m_t}} {x I_{m_t}} } \right). \] The pair $(\H_n,X_{\H_n})$ defines a Shimura datum of reflex field $\Q$.
Denote by $\Sh_{\H_n}$ the corresponding Shimura variety of dimension $\sum_i \binom{n_i + 1}{2}$. 

If $V \subseteq \H_n(\Af)$ is a fibre product (over the similitude characters) $V_1 \times_{\Af^\times} \cdots  \times_{\Af^\times} V_t $ of sufficiently small subgroups, we have
\[ \Sh_{\H_n}(V) = \Sh_{\GL_2}(V_1) \times_{\Gm} \cdots \times_{\Gm} \Sh_{\GSp_{2m_t,F_t}^*}(V_t), \]
where $\times_{\Gm}$ denotes the fibre product over the zero dimensional Shimura variety of level $D = {\rm det}(V_1) = \hdots = \nu(V_t)$ 
\[\pi_0(\Sh_{\GL_2})(D)=\hat{\Z}^\times/D.\]
We recall the reader that the complex points of $\Sh_{\H_n}(V)$ are given by
\[ \Sh_{\H_n}(V)(\C) = \H_n(\Q) \backslash \H_n(\A) / \Z_{\H_n}(\R) K_{\H_n, \infty} V, \]
where $\Z_{\H_n}$ denotes the center of $\H_n$ and $K_{\H_n, \infty} \subseteq \H_n(\R)$ is the maximal compact defined as the product $\mathrm{U}(1) \times \mathrm{U}(m_2)^{\delta_2} \ldots \times \mathrm{U}(m_t)^{\delta _t}$ if $\epsilon_n=1$ and defined as the product $\mathrm{U}(1) \times \mathrm{U}(1) \times \mathrm{U}(m_3)^{\delta_3} \ldots \times \mathrm{U}(m_t)^{\delta _t}$

Notice that the embedding $\iota_n : \H_n \to \G_n$ induces another Shimura datum $(\G_n,X_{\G_n})$ of reflex field $\Q$. 
For any neat open compact subgroup $U$ of $\G_n(\Af)$, denote by $\Sh_{\G_n}(U)$ the associated Shimura variety of dimension $d_n : = \tfrac{n(n+1)}{2}$.
We also write $\iota_n: \Sh_{\H_n}(U \cap \H_n) \hookrightarrow \Sh_{\G_n}(U)$ the closed embedding of codimension $c_n = d_n - \sum_i \binom{n_i + 1}{2} = \frac{1}{2}( d_n +1 - \epsilon_n)$ induced by the group homomorphism $\iota_n: \H_n \hookrightarrow \G_n$.

\subsection{Motivic cohomology classes for $\GSp_{2n}$}\label{Constructiongeneral}

\subsubsection{Modular units and Eisenstein series}\label{sectionmodularunits}
The inputs of our construction are the modular units already considered by Beilinson and Kato, which are related to real analytic Eisenstein series by the Kronecker limit formula.

Let $\mathbf{T}_2$ denote the diagonal maximal torus of $\GSp_2=\GL_2$ and let $\mathbf{B}_2$ denote the standard Borel. Define the algebraic character $\lambda: \mathbf{T}_2 \rightarrow \mathbf{G}_m$ by $\lambda(\mathrm{diag}(t_1,t_2))=t_1/t_2$.
Let $\Sc(\A^2)$ denote the space of Schwartz-Bruhat functions $\Phi$ on $\A^2$. Given $\Phi \in \Sc(\A^2)$, denote by \[f(g,\Phi,s):= |\operatorname{det}(g)|^s \int_{\GL_1(\A)} \Phi((0,t)g)|t|^{2s}d^\times t  \]
the normalised Siegel section in $\operatorname{Ind}_{\B_2(\A)}^{\GL_2( \A)}(|\lambda|^s)$ and define the associated Eisenstein series \begin{eqnarray}\label{eisensteinseriesgl2} {\rm E}(g,\Phi,s) := \sum_{\gamma \in \mathbf{B}_2(\Q) \backslash \GL_2(\Q)} f(\gamma g,\Phi ,s).  \end{eqnarray}

Fix the Schwartz-Bruhat function $\Phi_\infty$ on $\R^2$ defined by $(x,y) \mapsto e^{- \pi(x^2 + y^2)}$ and, for each $\overline{\Q}$-valued function $\Phi_f \in \Sc(\Af^2,\overline{\Q})$, the smallest positive integer $N_{\Phi_f}$ such that $\Phi_f$ is constant modulo $N_{\Phi_f}\widehat{\Z}^2$. Finally, denote $ \Sc_0(\Af^2,\overline{\Q})\subset \Sc(\Af^2,\overline{\Q})$ the space of elements $\Phi_f$ such that $\Phi_f((0,0))=0$. 

We now state the following (classical) result, which relates modular units to values of the adelic Eisenstein series defined in \eqref{eisensteinseriesgl2}.
 
\begin{proposition} \label{KLF} Let $\Phi_f \in \Sc_0(\Af^2,\overline{\Q})$ with $N_{\Phi_f} \geq 3$, then there exists $$
u(\Phi_f) \in \mathcal{O}(\Sh_{\GL_2}(K(N_{\Phi_f})))^\times \otimes \overline{\Q}$$ such that for any $g \in \GL_2(\A)$ we have
\[E(g,\Phi,s)={\rm log}|u(\Phi_f)(g)| + O(s), \] where $\Phi=\Phi_{\infty} \otimes \Phi_f$.
\end{proposition}

\begin{proof}
This is the second statement of \cite[Corollary 5.6]{PollackShahU21}, where $\nu_1$ is taken to be the trivial character.
\end{proof}

\begin{example}

As explained in \cite[Example 2.2]{CLRGSp6}, when $\Phi_f= {\rm char}((0,1)+N\widehat{\Z}^2)$ for $N \geq 4$, the corresponding $u(\Phi_f) \in \O(\Sh_{\GL_2}(K(N)))^\times \otimes \Q$ is given by $\prod_{b \in ( \Z/N\Z)^\times} g_{0, b /N}^{\varphi(N)}$, where $g_{0, \star/N}$ is the  Siegel unit as in \cite[\S 1.4]{Kato}. 

\end{example}

\subsubsection{The construction}   Let
$$\mathrm{Eis}_\mathcal{M}: \mathcal{S}_0(\A_f, \overline{\Q}) \rightarrow H^1_\mathcal{M}(\Sh_{\GL_2}, \overline{\Q}(1)) \simeq \mathcal{O}(\Sh_{\GL_2})^\times \otimes_{\Z} \overline{\Q}$$ be the $\GL_2(\A_f)$-equivariant map define by $\Phi_f \mapsto u(\Phi_f)$, where $H^1_\mathcal{M}(\Sh_{\GL_2}, \overline{\Q}(1))$ denotes $\underrightarrow{\lim}_{V} H^1_\mathcal{M}(\Sh_{\GL_2}(V), \overline{\Q}(1))$ and $\mathcal{O}(\Sh_{\GL_2})^\times \otimes_{\Z} \overline{\Q}$ denotes $\underrightarrow{\lim}_{V} (\mathcal{O}(\Sh_{\GL_2}(V))^\times \otimes_{\Z} \overline{\Q})$, the limits being taken over all neat compact open subgroups $V \subset \GL_2(\A_f)$.
 
In the case $\epsilon_n=1$, let $$V_1 \subset \GSp_2(\A_f), V_2 \subset \GSp_{2m_2, F_2}^*(\A_f), \ldots, V_t \subset \GSp_{2m_t, F_t}^*(\A_f)$$
denote neat compact open subgroups. If $\epsilon_n=2$ we make a similar choice, adapting the notation in an obvious way. We assume that the images of the $V_s$ by the similtude characters are the same. Taking the fiber products over the similitude character, we obtain a compact open subgroup $V=V_1 \times_{\A_f^\times} \ldots \times_{\A_f^\times} V_s$ of $ \H_n(\A_f)$. Let $U \subset \G_n(\A_f)$ be a neat compact open subgroup such that the embedding $\iota_n$ induces a closed embedding $\Sh_{\H_n}(V) \hookrightarrow \Sh_{\G_n}(U)$. The Shimura variety $\Sh_{\H_n}(V)$ is of codimension $c_n = d_n -\sum_i \binom{n_i+1}{2}=\frac{1}{2}(d_n+1-\epsilon_n)$ in $\Sh_{\G_n}(U)$. As a consequence, we have an induced map on motivic cohomology 
\[ \iota_{n\,*}: H^{\epsilon_n}_{\mathcal{M}} \big(\Sh_{\H_n}(V), \overline{\Q}(\epsilon_n) \big) \to H^{d_n + 1}_{\mathcal{M}}\big( \Sh_{\G_n}(U), \overline{\Q}(t_n) \big) \] 
where $t_n=\epsilon_n+c_n$. For any $n$, the projection on the first factor of $\Sh_{\H_n}(V)$ is a morphism $p_1: \Sh_{\H_n}(V) \rightarrow \Sh_{\GL_2}(V_1)$. Hence when $n$ is such that $\epsilon_n=1$, we have the sequence of morphisms
$$
\begin{CD}
\mathcal{S}_0(\A_f, \overline{\Q})^{V_1} @>\mathrm{Eis}_\mathcal{M}>> H^1_\mathcal{M}(\Sh_{\GL_2}(V_1), \overline{\Q}(1))\\
@>p_1^*>> H^1_{\mathcal{M}} \big(\Sh_{\H_n}(V), \overline{\Q}(1) \big)\\
@>\iota_{n\,*}>> H^{d_n + 1}_{\mathcal{M}}\big( \Sh_{\G_n}(U), \overline{\Q}(t_n) \big).
\end{CD}
$$

\begin{definition} \label{eis1}
We define
$
\mathrm{Eis}_{\mathcal{M}, n}: \mathcal{S}_0(\A_f, \overline{\Q})^{V_1} \rightarrow H^{d_n + 1}_{\mathcal{M}}\big( \Sh_{\G_n}(U), \overline{\Q}(t_n) \big)
$
to be the composite of these morphisms.
\end{definition}

When $\epsilon_n=2$, the projection on the second factor of $\Sh_{H_n}(V)$ is also a morphism $p_2: \Sh_{\H_n}(V) \rightarrow \Sh_{\GL_2}(V_1)$. Hence when $\epsilon_n=2$, we have the sequence of morphisms
$$
\begin{CD}
\mathcal{S}_0(\A_f, \overline{\Q})^{V_1} \otimes_{\overline{\Q}} \mathcal{S}_0(\A_f, \overline{\Q})^{V_2} @>\mathrm{Eis}_\mathcal{M} \otimes \mathrm{Eis}_\mathcal{M}>> H^1_\mathcal{M}(\Sh_{\GL_2}(V_1), \overline{\Q}(1)) \otimes H^1_\mathcal{M}(\Sh_{\GL_2}(V_2), \overline{\Q}(1))\\
 @>p_1^* \otimes p_2^*>> H^1_{\mathcal{M}} \big(\Sh_{\H_n}(V), \overline{\Q}(1) \big) \otimes H^1_{\mathcal{M}} \big(\Sh_{\H_n}(V), \overline{\Q}(1) \big) \\
@>\cup>> H^2_{\mathcal{M}} \big(\Sh_{\H_n}(V), \overline{\Q}(2) \big)\\
@>\iota_{n\,*}>> H^{d_n + 1}_{\mathcal{M}}\big( \Sh_{\G_n}(U), \overline{\Q}(t_n) \big),
\end{CD}
$$
where the third morphism is the cup-product in motivic cohomology. 

\begin{definition} \label{eis2}
When $\epsilon_n = 2$, we define
\[\mathrm{Eis}_{\mathcal{M}, n}: \mathcal{S}_0(\A_f, \overline{\Q})^{V_1} \otimes_{\overline{\Q}} \mathcal{S}_0(\A_f, \overline{\Q})^{V_2} \rightarrow H^{d_n + 1}_{\mathcal{M}}\big( \Sh_{\G_n}(U), \overline{\Q}(t_n) \big) \] to be the composite of these morphisms. 
\end{definition}

\begin{remark} The notation $\mathrm{Eis}_{\mathcal{M}, n}$ is slightly abusive as these morphisms depend also on $U$, $V$ and the data entering in the definition of $\iota_n$.
\end{remark}

In the rest of the paper, when no confusion arises, we simplify our notation by identifying $\H_n, \G_n, \epsilon_n, d_n$ and $t_n$ with $\H,\G, \epsilon ,  d$ and $t$.  

\section{Archimedean regulators}

We now turn to the study of the classes constructed in Definition \ref{eis1} and Definition \ref{eis2}. We first construct a harmonic differential form associated to a cuspidal automorphic representation and then we explain how one use this differential form to define a natural linear form on Deligne--Beilinson cohomology and calculate the image of the Deligne realisations of the motivic cohomology classes.

\subsection{Representation theory}

\subsubsection{Cartan decomposition and root system}\label{cartandec} Let $K_\infty$ be the maximal compact subgroup of $\G_0(\R)$, where $\G_0:=\Sp_{2n}$, which fixes the point $i J \in X_{\G}$. We have an isomorphism $\kappa :U(n) \simeq K_\infty$. Moreover, one has a Cartan decomposition \[ \mathfrak{g}_{0,\C}=\mathfrak{k}_{\C} \oplus \mathfrak{p}^+_{\C} \oplus \mathfrak{p}^-_{\C}, \]
where $\mathfrak{k} = {\rm Lie}(K_\infty)$ and $\mathfrak{p}^+_{\C}$, resp. $\mathfrak{p}^-_{\C}$, is the holomorphic, resp. anti-holomorphic, tangent space of $X_\G$ at $i J \in X_{\G}$. 

Let $T_\infty \subset K_\infty$ denote $\{\kappa(\mathrm{diag}(z_1, z_2, \ldots, z_n)), z_i \in \mathrm{U}(1)\}$. Then $T_\infty$ is Cartan subgroup of $K_\infty$ with Lie algebra $\mathfrak{h} \subset \mathfrak{k}$. Then $\mathfrak{h}$ is a compact Cartan subalgebra of $\mathfrak{g}_{0}$. Fix a basis $(e_j)_j$ of the dual $\mathfrak{h}^*_\C$. A system of positive roots for $(\mathfrak{g}_{0,\C}, \mathfrak{h}_\C)$ is then given by
\begin{eqnarray*}
2 e_j, &\;& 1 \leq j \leq n, \\
e_j + e_k, &\;& 1 \leq j < k \leq n, \\
e_j - e_k, &\;& 1 \leq j < k \leq n.
\end{eqnarray*}
The simple roots are $e_1 - e_2, \hdots e_{n - 1} - e_n, 2 e_n$. We note that $\mathfrak{p}^+_{\C}$ is spanned by the root spaces corresponding to the positive roots of type $2 e_j$ and $e_j+e_k$. We denote $\Delta = \{ \pm 2 e_j, \pm (e_j \pm e_k) \}$ the set of all roots, $\Delta_\mathrm{c} = \{ \pm(e_j - e_k)\}$ the set of compact roots and $\Delta_{\rm nc} = \Delta - \Delta_{\rm c}$ the non-compact roots. Finally, we note $\Delta^+, \Delta_{\rm c}^+$ and $\Delta_{\rm nc}^+$ the set of positive, positive compact and positive non-compact roots, respectively. We denote by $X_\alpha$ a root vector for any given root $\alpha$. Then, we have that $\mathfrak{p}^{\pm}_{\C}=\bigoplus_{\alpha \in \Delta_{\rm nc}^+} \C X_{\pm \alpha} $.

\subsubsection{Weyl groups} \label{weylgroups} Recall that the Weyl group of $\G_0$ is given by $\mathfrak{W}_{\G_0} = \{ \pm 1\}^n \rtimes \mathfrak{S}_n$. The reflection $\sigma_j$ in the orthogonal hyperplane of $2 e_j$ simply reverses the sign of $e_j$ while leaving the other $e_k$ fixed. The reflection $\sigma_{jk}$ in the orthogonal hyperplane of $e_j - e_k$ exchanges $e_j$ and $e_k$ and leaves the remaining $e_\ell$ fixed. The Weyl group $\mathfrak{W}_{K_\infty}$ of $K_\infty \cong U(n)$ is isomorphic to $\mathfrak{S}_n$ and, via the embedding into $\G_0$, identifies with the subgroup of $\mathfrak{W}_{\G_0}$ generated by the $\sigma_{jk}$.

\subsubsection{$K_\infty$-types}\label{kinftypes}

We previously defined the maximal compact subgroup $\kappa :U(n) \simeq K_\infty$ of $\G_0(\R)$, with Lie algebra $\mathfrak{k}$ and the Cartan subgroup  $T_\infty$ of $K_\infty$ with Lie algebra $\mathfrak{h} \subset \mathfrak{k}$. Its group of algebraic characters is isomorphic to $\Z^n$ via $(k_1, k_2,\ldots, k_n) \mapsto \lambda(k_1, k_2,\ldots, k_n)$, where
\[ \lambda(k_1, k_2,\ldots, k_n) : \kappa(\mathrm{diag}(z_1, z_2, \ldots, z_n)) \mapsto z_1^{k_1} z_2^{k_2} \cdots  z_n^{k_n}. \]
An algebraic character is dominant for our choice of $\Delta_{\rm c}^+$ if $k_1 \geq k_2 \geq \cdots \geq k_n$.  For any dominant integral weight $\lambda$, there exists a unique (up to isomorphism) irreducible representation $\tau_{\lambda}$ of $K_\infty$ in a finite dimensional $\C$-vector space of highest weight $\lambda$ and every irreducible representation of $K_\infty$ is obtained in this way (up to isomorphism). In what follows, we will denote the irreducible representation of highest weight $\lambda(k_1, k_2,\ldots, k_n) $ by $\tau_{(k_1, \cdots, k_n)}$.

\subsection{Test vectors}
\subsubsection{Discrete series \emph{L}-packets} \label{discreteseries}

We recall some standard facts on discrete series. For any non-singular weight $\Lambda \in \Delta$, define
\[ \Delta^+(\Lambda) := \{ \alpha \in \Delta \; : \; \langle \alpha, \Lambda \rangle > 0 \}, \;\;\; \Delta^+_c(\Lambda) = \Delta^+(\Lambda) \cap \Delta_c, \]
where $\langle \;,\;\rangle$ is the standard scalar product on $\R^n$.

Let $\lambda$ be a dominant weight for $\G_0$ (with respect to the compact torus $T_\infty$) and let $\rho = \frac{1}{2} \sum_{\alpha \in \Delta^+} \alpha = (n, n-1, \hdots, 1)$. As $| \mathfrak{W}_{\G_0} / \mathfrak{W}_{K_\infty}| = 2^n$, the set of equivalence classes of irreducible discrete series representations of $\G_0(\R)$ with Harish-Chandra parameter $\lambda + \rho$ contains $2^n$ elements. More precisely, let us choose representatives $\{w_1, \ldots, w_{2^n}\}$ of $\mathfrak{W}_{\G_0} / \mathfrak{W}_{K_\infty}$ of increasing length and such that for any $1 \leq i \leq 2^n$, the weight $w_i (\lambda + \rho)$ is dominant for $K_\infty$. Then for any $1 \leq i \leq 2^n$ there exists an irreducible discrete series $\pi_\infty^{\Lambda}$, where $\Lambda = w_i(\lambda + \rho)$, of Harish-Chandra parameter $\Lambda$ and containing with multiplicity $1$ the minimal $K_\infty$-type with highest weight $ \Lambda + \delta_{\G_0} - 2\delta_{K_\infty}$ where $\delta_{\G_0}$, resp. $\delta_{K_\infty}$, is the half-sum of roots, resp. of compact roots, which are positive with respect to the Weyl chamber in which $\Lambda$ lies, i.e.,
\[ 2 \delta_{\G_0} := \sum_{\alpha \in \Delta^+(\Lambda)} \alpha, \;\;\; 2 \delta_{K_\infty} := \sum_{\alpha \in \Delta^+_c(\Lambda)} \alpha. \]
Moreover, for $i \neq j$, $\Lambda = w_i (\lambda + \rho)$, $\Lambda' = w_j(\lambda + \rho)$, the representations $\pi_\infty^{\Lambda}$ and $\pi_\infty^{\Lambda'}$ are not equivalent and any discrete series of $\G_0$ is obtained in this way (\cite[Theorem 9.20]{knapp}). We define the discrete series $L$-packet $P(V^\lambda)$ associated to $\lambda$ to be the set of isomorphism classes of discrete series of $\G_0(\R)$ whose Harish-Chandra parameter is of the form $\Lambda = w_i(\lambda + \rho)$, for some $1 \leq i \leq 2^n$.

The discrete series $L$-packets for  $\G(\R)$ can be described similarly, but the set of its Harish-Chandra parameters changes slightly. This is due to the fact that its maximal compact subgroup has two connected components and the set of parameters has to be considered up to the action of $\mathfrak{W}_{K_\infty}$ and of $w_{2^n}$. Indeed, the Weyl element $w_{2^n}$, which is the anti-diagonal matrix with all entries $-1$, now belongs to the connected component away from the identity of the maximal compact subgroup. This translates into identifying any parameter  $( \lambda_1, \lambda_2, \ldots , \lambda_n )$ with $( -\lambda_n, -\lambda_{n-1}, \ldots , -\lambda_1 )$.  Let $\Pi^{triv}$ denote the discrete series $L$-packet for $\G(\R)$ associated to the trivial representation. It's a set with $2^{n-1}$ elements and if $\pi_\infty \in \Pi^{triv}$ is a discrete series for $\G(\R)$, then its restriction to $\G_0(\R)$ decomposes as $\pi_\infty^1 \oplus \pi_\infty^2$, where $\pi_\infty^i$ is a discrete series for $\G_0(\R)$ in the $L$-packet associated to the trivial representation, and they are both conjugate one to the other.

\subsubsection{Lie algebra cohomology}

Let $A_\G = \R^\star_+$ denote the identity component of the center of $\G(\R)$ and let $K_\G =A_\G K_\infty \subset \G(\R)$. The embedding $\mathfrak{g}_{0,\C}\subset \mathfrak{g}_\C$ induces an isomorphism \[\mathfrak{g}_{0,\C}/ \mathfrak{k}_\C \simeq  \mathfrak{g}_\C / ({\rm Lie}(K_\G))_\C.\]  
By \cite[II. Proposition 3.1]{BorelWallach}, for any discrete series $\pi_\infty \in \Pi^{triv}$, we have
\[ H^d(\mathfrak{g}, K_\G; \pi_\infty) = {\rm Hom}_{K_\infty}(\bigwedge^d \mathfrak{g}_{0,\C}/ \mathfrak{k}_\C, \pi_\infty), \]
where $d=\frac{n(n+1)}{2}$. Using the decomposition ${\pi_{\infty}}_{|_{\G_0(\R)}}=\pi_\infty^1 \oplus \pi_\infty^2$, we further have \[H^d(\mathfrak{g}, K_\G; \pi_\infty) = {\rm Hom}_{K_\infty}(\bigwedge^d \mathfrak{g}_{0,\C}/ \mathfrak{k}_\C, \pi_\infty^1) \oplus {\rm Hom}_{K_\infty}(\bigwedge^d \mathfrak{g}_{0,\C}/ \mathfrak{k}_\C, \pi_\infty^2). \]
By \cite[Theorem II.5.3]{BorelWallach}, each space \[ {\rm Hom}_{K_\infty}\left(\bigwedge^d \mathfrak{g}_{0, \C}/ \mathfrak{k}_{\C}, \pi_\infty^{i}\right)\]
has dimension 1.
This is a consequence of the fact (cf. the proof of \cite[Theorem II.5.3]{BorelWallach}) that the minimal $K_\infty$-type of $\pi_\infty^{i}$ appears uniquely in $\bigwedge^d \mathfrak{g}_{0, \C}/ \mathfrak{k}_{\C}$. Thus, for any $\pi_\infty \in \Pi^{triv}$, $ H^d(\mathfrak{g}, K_\G; \pi_\infty)$ has dimension 2. By using the Cartan decomposition, we get \[\bigwedge^d \mathfrak{g}_{0,\C}/ \mathfrak{k}_\C = \bigoplus_{p+q=d} \bigwedge^p \mathfrak{p}^+_{\C} \otimes_\C  \bigwedge^q \mathfrak{p}^-_{\C}.\] 
Hence, there exists a unique pair $(p_i,q_i)$ such that
$\Hom_{K_\infty}\left( \bigwedge^{p_i} \mathfrak{p}_{\C}^+ \otimes \bigwedge^{q_i} \mathfrak{p}_{\C}^-, \pi_\infty^i \right)$ is non-zero
and hence of dimension one. Concretely, $p_i$ (resp. $q_i$) is the number of positive non-compact roots in $\Delta^+(\Lambda)$ (resp. $\Delta^-(\Lambda)$), where $\Lambda$ is the Harish-Chandra parameter of $\pi_\infty^i$. We call such a pair $(p_i,q_i)$  the Hodge type of $\pi_\infty^i$. Since $\pi_\infty^2 \simeq \overline{\pi}_\infty^1$, we have that $(p_1,q_1)=(q_2,p_2)$ and therefore we can associate to $\pi_\infty \in \Pi^{triv}$ the Hodge type $(p_1,q_1)$ (counted up to complex conjugation).

\subsubsection{Test vectors}

Let $\pi=\pi_\infty \otimes \pi_f$ be a cuspidal automorphic representation of $\G(\A)$ with trivial central character and with archimedean component in the discrete series $L$-packet $\Pi^{triv}$.

\begin{lemma} \label{testvector} Let $p,q \geq 0$ be two integers such that $p+q=d$. Suppose that $\pi_\infty|_{\G_0(\R)} \simeq \pi_\infty^1 \oplus \overline{\pi}_\infty^1$ with $\Hom_{K_\infty}\left(\bigwedge^{p} \mathfrak{p}^+_{\C} \otimes_\C \bigwedge^{q} \mathfrak{p}^-_{\C}, \pi_\infty^1 \right) \neq 0$ and that $\pi_f^U \neq 0$. Let $\Psi = \Psi_\infty \otimes \Psi_f$ be a cusp form in the space of $\pi$ such that $\Psi_\infty$ is a highest weight vector of the minimal $K_\infty$-type $\tau_\infty^1$ of $\pi_\infty^1$ and $\Psi_f$ is a non-zero vector in $\pi_f^U$. Let $X_\infty^1$ be a highest weight vector in the $K_\infty$-type $\tau_\infty^1 \subset \bigwedge^{p} \mathfrak{p}^+_{\C} \otimes_\C \bigwedge^{q} \mathfrak{p}^-_{\C}$ (this inclusion is assured by the hypothesis that $\pi_\infty$ contributes to the $(p, q)$-part of the cohomology). Then there exists up to scalars a unique non-zero harmonic $(p, q)$ differential form 
$$
\omega_\Psi \in \Hom_{K_{\infty}}\left(\bigwedge^{p} \mathfrak{p}^+_{\C} \otimes_\C \bigwedge^{q} \mathfrak{p}^-_{\C}, \pi_\infty^1 \right) \otimes  \pi_f^U
$$
on $\Sh_{\G}(U)$ such that $\omega_\Psi(X_\infty^1) = \Psi$. Moreover, the cohomology class of $\omega_\Psi$ belongs to $H^{d}_{dR,!}(\Sh_{\G}(U), \C)$.
\end{lemma}

\begin{proof}
The results follows basically from \cite[Theorem II.5.3]{BorelWallach}. The fact that the harmonic form $\omega_\Psi \in H^{d}_{dR,!}(\Sh_{\G}(U), \C)$ follows from the cuspidality of $\Psi$ and from \cite[Corollary 5.5]{borel2}.
\end{proof}

\subsection{Deligne--Beilinson cohomology} \label{sectiondelignecohomology}

Let $X$ denote a complex analytic variety which is smooth, quasi-projective and of pure dimension $d$. Let $\overline{X}$ be a smooth compactification of $X$ such that $D=\overline{X}-X$ is a simple normal crossing divisor. We denote by $j: X \rightarrow \overline{X}$ the open embedding. We will assume that $X$ is defined as the analytification of the base change to $\C$ of a smooth, quasi-projective $\R$-scheme. The complex conjugation $F_\infty$ is an antiholomorphic involution on $X$. For $p \in \Z$, let $\R(p)$ denote the subgroup $(2\pi i)^p \R$ of $\C$. We will denote by the same symbol the constant sheaf with value $\R(p)$ on $X$. Let $\Omega^*_X$ be the sheaf of holomorphic differential forms on $X$ and let $\Omega^*_{\overline{X}}(\log D)$ be the sheaf of holomorphic differential forms on ${X}$ with logarithmic poles along $D$ (see \cite[\S 3.1]{Deligne-HodgeII}). The Hodge filtration on $\Omega^*_{\overline{X}}(\log D)$ is defined as $F^p\Omega^*_{\overline{X}}(\log D)=\bigoplus_{p' \geq p} \Omega^{p'}_{\overline{X}}(\log D)$. There are natural quasi-isomorphisms of complexes $Rj_*\C \rightarrow Rj_*\Omega^*_X$ and $\Omega^*_{\overline{X}}(\log D) \rightarrow Rj_* \Omega^*_X$ (see \cite{Deligne-HodgeII} or \cite{Jannsen} for the basic facts used here). \\

Deligne--Beilinson cohomology is defined as
\begin{equation} \label{defBDcohom}
R \Gamma_{\mathcal{D}}(X, \R(p)) =  \mathrm{cone} \left( R \Gamma(X, \R(p)) \oplus F^p R \Gamma(X, \C) \to R \Gamma(X, \C) \right) [-1],
\end{equation}
where the map is the difference of the natural maps. Let $\overline{F_\infty^*} = F_\infty^* \otimes c$ be the de Rham involution given by the action of the complex conjugation on $X$ and on the coefficients. We define real Deligne--Beilinson cohomology as
\[ R \Gamma_\mathcal{D}(X / \R, \R(p)) = R \Gamma_\mathcal{D}(X, \R(p))^{\overline{F_\infty^*}}, \]
Real and complex Deligne--Beilinson cohomology groups $H^n_\mathcal{D}(X / \R, \R(p))$ and $H^n_\mathcal{D}(X, \R(p))$ are defined to be the cohomology groups of the complexes $R \Gamma_{\mathcal{D}}(X / \R, \R(p))$ and $R \Gamma_{\mathcal{D}}(X, \R(p)),$ respectively. By definition, one can calculate these cohomology groups using any complexes which are quasi-isomorphic to the terms in \eqref{defBDcohom}, e.g., via the hypercohomology groups of the complex
$$
\R(p)_{\mathcal{D}} := \text{cone}(Rj_*\R(p) \oplus F^p\Omega^*_{\overline{X}}(\log D) \rightarrow Rj_* \Omega^*_X)[-1].
$$

\subsection{Hodge structures on the cohomology of a smooth variety}

Let us first introduce some objects. Let $\mathscr{A}^*_{\R}$ be the complex of sheaves of real valued smooth differential forms on $\overline{X}$ and denote by $\mathscr{A}^*_{\R}(\overline{X})$ the complex of vector spaces of global sections. We denote by $\mathscr{A}^*$, resp. $\mathscr{A}^*(\overline{X})$, the complex of sheaves of complex valued smooth differential forms on $\overline{X}$, resp. the complex of its global sections. Observe that, since $\mathscr{A}^*_{\R}$ (resp. $\mathscr{A}^*$) is a module over the sheaf of smooth real (resp. complex) valued functions on $\overline{X}$, it has partitions of unity and hence it is a complex of fine sheaves.

Let $\mathscr{A}^0_{si} \subseteq \j_* \mathscr{A}^0$ be the sheaf on $\overline{X}$ of functions defined over $X$ which are slowly increasing along $D$. Precisely, recall that around any point of $\overline{X}$ one can find a coordinate system $(z_1, \hdots, z_d)$ and an integer $0 \leq n \leq d$ such that $\overline{X}$ is locally isomorphic to a polydisc and such that $D$ is defined by the equation $ z_1 \hdots z_n = 0$, so that $X$ is locally isomorphic to $(\Delta_r^\times)^n \times \Delta_r^{d - n}$, with $\Delta_r^\times = \{ z \in \C, 0 < |z| < r\}$ the punctered disc of radius $r$ and $\Delta_r = \{ z\in\C, |z| < r \}$ the complete disc of radius $r$. A complex valued function $f$ on $X$ is slowly increasing (resp. rapidly decreasing) if locally it satisfies
\[ |f(z)| \leq C \prod_{i = 1}^n |\log|z_i||^N \]
for some $N \geq 0$ (resp. for all $N \leq 0$) and some constant $C$. As it is usual, we impose similarly growth conditions on some derivatives of $f$, cf. \cite[\S 4.1]{CLRGSp6} for these precise conditions which will not be needed here. We will denote by $\mathscr{A}^*_{si, \R}$ and $\mathscr{A}^*_{rd, \R}$ the complexes of sheaves on $\overline{X}$ of real-valued slowly increasing and rapidly decreasing differential forms. We also let $\mathscr{A}^*_{si} $ and $\mathscr{A}^*_{rd}$ the complexes of sheaves on $\overline{X}$ of complex-valued slowly increasing and rapidly decreasing differential forms.  These are complexes of fine sheaves and we denote by $\mathscr{A}^*_{si}(\overline{X})$ and $\mathscr{A}^*_{rd}(\overline{X})$ the corresponding complexes of global sections. The complex structure on $X$ induces bigradings
\[ \mathscr{A}_{si}^* = \bigoplus_{p, q} \mathscr{A}^{p, q}_{si}, \]
\[ \mathscr{A}_{rd}^* = \bigoplus_{p, q} \mathscr{A}^{p, q}_{rd}. \]
One has $\mathscr{A}^{p, q}_{si} = \mathscr{A}^{p, q}_{\overline{X}}(\log \, D) \otimes_{\mathscr{A}^0_{\overline{X}}(\log \, D)} \mathscr{A}^0_{si}$ and idem for rapidly decreasing, where $\mathscr{A}_{\overline{X}}(\log \, D)$ denotes the sheaf of smooth differential forms with logarithmic singularities along $D$ as defined in \cite{burgosCinfty}. The space of rapidly decreasing differential forms is naturally equipped with a Fr\'echet topology and we define tempered currents $\mathscr{D}^{p, q}$ to be the sheaf $U \mapsto \Gamma_c(U, \mathscr{A}^{d-p, d-q}_{rd})^\vee$ on $\overline{X}$ of continuous linear forms on $\mathscr{A}_{rd}^{d-p, d-q}$, where $U \subseteq \overline{X}$. We let
\[ \mathscr{D}^{*} = \bigoplus_{p, q} \mathscr{D}^{p, q}.\] This is again a complex of fine sheaves and we denote by $\mathscr{D}^*(\overline{X})$ the corresponding complex of global sections. To any slowly increasing smooth differential form $\phi \in \mathscr{A}_{si}^{p, q}(U)$ one can associate a tempered current $T_\phi \in \mathscr{D}^{p,q}(U)$ by the formula \[ {T}_\phi(\eta)=\frac{1}{(2\pi i)^d}\int_{X} \phi \wedge \eta , \;\;\;\;  \eta \in \mathscr{A}_{rd}^{d - p, d - q}(U). \] The induced map $\mathscr{A}^*_{si} \to \mathscr{D}^*$ is obviously injective. The key result concerning these complexes is the following.

\begin{theorem}[{\cite[Theorem 1.6]{CLRGSp6}}] \label{TheoHScohomo}
The natural inclusions
\[ (\Omega_{\overline{X}}(\log \, D), F) \to (\mathscr{A}_{\overline{X}}(\log \, D), F) \to (\mathscr{A}_{si}, F) \to (\mathscr{D}^*, F) \]
are filtered quasi-isomorphisms. Moreover, the last two quasi-isomorphisms are compatible for the corresponding real structures.
\end{theorem}

\subsection{Deligne--Beilinson cohomology II}

We now use the results above to give different useful descriptions and properties of Deligne--Beilinson cohomology.

\begin{proposition}[{\cite[Proposition 4.22]{CLRGSp6}}] \label{DBcoh-diff-form}
There is a quasi-isomorphism
\[ R \Gamma_\mathcal{D}(X, \R(p)) \simeq \mathrm{cone}(F^p \mathscr{A}^*_{si}(\overline{X}) \to \mathscr{A}^*_{si, \R(p-1)}(\overline{X}))[-1], \]
where the arrow is induced by the projection $\pi_{p-1} : \C \to \R(p-1)$ defined by $\pi_{p - 1}(z) = \frac{z + (-1)^{p-1} \overline{z}}{2}$. In particular, we have canonical isomorphisms
\begin{equation}\label{delignecoh0}
H^n_{\mathcal{D}}(X, \R(p)) \simeq \frac{\{ (\phi, \phi') \in F^p \mathscr{A}^n_{si}(\overline{X}) \oplus \mathscr{A}^{n-1}_{si, \R(p-1)}(\overline{X}) \,|\, d\phi = 0, d \phi' = \pi_{p-1}(\phi) \}}{d(\tilde{\phi}, \tilde{\phi}')},
\end{equation}
\end{proposition}

\begin{remark} \label{unitDeligne} Let
$r_{1,1}: H^1_{\mathcal{M}}(X, \Q(1)) \rightarrow H^1_{\mathcal{D}}(X, \R(1))$ be Beilinson's regulator. Recall the canonical isomorphism $\mathcal{O}(X)^\times \otimes \Q \simeq H^1_{\mathcal{M}}(X, \Q(1))$. Then for $u \in \mathcal{O}(X)^\times$, the Deligne--Beilinson cohomology class $r_{1,1}(u \otimes 1)$ is represented by $(d \log(u), \log|u|) \in F^1 
\mathscr{A}^1_{si}(\overline{X}) \oplus \mathscr{A}^0_{si}(\overline{X}) \otimes \R$.
\end{remark}

The following result gives the explicit description of the external cup-product in Deligne--Beilinson cohomology via the isomorphism of Proposition \ref{DBcoh-diff-form}.

\begin{proposition} \label{products} Let $X$ and $X'$ be the base changes to $\C$ of two smooth, quasi-projective $\R$-schemes. Let  $p_X: X \times X' \longrightarrow X$ and $p_{X'}: X \times X' \longrightarrow X'$ be the canonical projections. Then, via the isomorphism of Proposition \ref{DBcoh-diff-form}, the external cup-product
$$
\sqcup: H^{m}_\mathcal{D}(X/\R, \R(m)) \otimes H^{m'}_\mathcal{D}(X'/\R, \R(m')) \longrightarrow H^{m+m'}_\mathcal{D}(X \times X'/\R, \R(m+m'))
$$
is
$$
(\phi, \omega) \sqcup (\phi', \omega')=(p_X^*\phi \wedge p_{X'}^*( \pi_{m'} \omega')+(-1)^m p_X^*(\pi_m \omega) \wedge p_{X'}^*\phi', p_X^* \omega \wedge p_{X'}^* \omega')
$$
for any $m, m'$.
\end{proposition}

\begin{proof} The external cup-product is by definition $x \sqcup  x' = p_X^*(x) \cup p_{X'}^*(x')$, where $\cup$ denotes the usual cup-product. Hence, the statement follows from the explicit formulas for the usual cup-product given in \cite[\S 2.5]{DeningerScholl} (see also \cite[\S 3.10]{EsnaultViehweg}).
\end{proof}

The following description will be key in our arguments.

\begin{proposition} [{\cite[Theorem 4.24]{CLRGSp6}}] \label{DBcoho-currents}
We have
\[ R \Gamma_\mathcal{D}(X, \R(p)) = \mathrm{cone} \left( F^p \mathscr{D}^*(\overline{X}) \to \mathscr{D}_{\R(p-1)}^*(\overline{X})  \right)[-1], \]
In particular
\[ H^n_{\mathcal{D}}(X, \R(p)) = \frac{\{ (S, T) : dS = 0, dT = \pi_{p-1}(S) \}}{d(\widetilde{S}, \widetilde{T})}, \]
where $(S, T) \in F^p \mathscr{D}^{n}(\overline{X}) \oplus \mathscr{D}^{n-1}_{\R(p - 1)}(\overline{X})$ and $d(\widetilde{S}, 
\widetilde{T}) = (d S, dT - \pi_{p - 1}(S))$.
\end{proposition}

\begin{proof}
This follows from Theorem \ref{TheoHScohomo}.
\end{proof}

In what follows, for $(S, T) \in F^p \mathscr{D}^{n}(\overline{X}) \oplus \mathscr{D}^{n-1}_{\R(p - 1)}(\overline{X})$ such that $dS=0$ and $dT=\pi_{p-1}(S)$, we will denote by $[(S,T)] \in H^n_{\mathcal{D}}(X, \R(p))$ the cohomology class of the pair $(S,T)$.

\begin{proposition} [{\cite[Proposition 4.25]{CLRGSp6}}]\label{compatibility} Let $x \in H^n_\mathcal{D}(X / \R, \R(n))$ be Deligne--Beilinson cohomology class which is represented, via the isomorphism of Proposition \ref{DBcoh-diff-form}, by a pair $(\phi, \phi')$ of smooth slowly increasing differential forms. Then via the isomorphism of Proposition \ref{DBcoho-currents}, the class $x$ is represented by the pair of currents $(T_\phi, T_{\phi'})$.
\end{proposition}

We also need to recall the functoriality of Deligne--Beilinson cohomology for proper morphisms. Let $f: X' \rightarrow X$ be a proper morphism of pure relative codimension $c$. Let $\overline{X'}$ denote a smooth compactification of $X'$ such that $D'=\overline{X'}-X'$ is a simple normal crossing divisor. Assume further that $f$ extends to a morphism $\overline{X'} \rightarrow \overline{X}$ that we still denote by $f$ and such that $f^{-1}(D)=D'$. By Poincar\'e duality between Deligne--Beilinson cohomology \cite[Theorem 1.15]{Jannsen} and homology and covariance of Deligne--Beilinson homology by proper maps, one has a functorial map
\[ f_*: H^n_{\mathcal{D}}(X', \R(p)) \rightarrow H^{n+2c}_{\mathcal{D}}(X, \R(p+c)). \]
If $T \in D_{\overline{X}'}^{p, q}$ is a tempered current on $X'$ then the formula 
\[ f_* T (\omega) = T(f^* \omega) \]
defines an element $f_* T \in D_{\overline{X}}^{p + c, q + c}$.

\begin{proposition} [{\cite[Proposition 4.26]{CLRGSp6}}]\label{push} Via the isomorphism of Proposition \ref{DBcoho-currents}, we have
\[ f_*([(T, T')]) = [(f_*T, f_*T')]. \]
\end{proposition}

\begin{remark} \label{functorial} We would like to apply Proposition \ref{push} to the closed embedding $\iota_n: \Sh_{\H_n}(U \cap \H_n) \hookrightarrow \Sh_{\G_n}(U)$. To this end, we need to extend $\iota_n$ to a morphism $\overline{\iota}_n: \overline{\Sh_{\H_n}(U \cap \H_n)} \rightarrow \overline{\Sh_{\G_n}(U)}$ between some smooth projective compactifications of $\Sh_{\H_n}(U \cap \H_n)$ and $\Sh_{\G_n}(U)$. This is possible according to \cite[Proposition 3.4]{harris-functorial}.
\end{remark}

Finally, we have the following explicit description of the cup product.

\begin{lemma}\label{regulator-currents} Let $r_\mathcal{D}: H^{d + 1}_{\mathcal{M}}\big( \Sh_{\G}(U), \overline{\Q}(t) \big) \rightarrow H^{d + 1}_{\mathcal{D}}\big( \Sh_{\G}(U), \R(t) \big) \otimes_{\Q} \overline{\Q}$ be Beilinson's higher regulator. Let $  \Phi_f \in  \mathcal{S}_0(\A_f, \overline{\Q})^{V_1}$ (resp. $ \Phi_{1, f} \otimes \Phi_{2, f} \in  \mathcal{S}_0(\A_f, \overline{\Q})^{V_1} \otimes_{\overline{\Q}}  \mathcal{S}_0(\A_f, \overline{\Q})^{V_2}$) if $\epsilon = 1$ (resp. $\epsilon = 2$). Then, via the isomorphisms given by Proposition \ref{DBcoho-currents}, the cohomology class $r_{\mathcal{D}}(\mathrm{Eis}_{\mathcal{M},n}(\Phi_f))$ (resp. $r_{\mathcal{D}}(\mathrm{Eis}_{\mathcal{M},n}(\Phi_{1, f} \otimes \Phi_{2, f}))$) is represented by the pair of currents $(\iota_{n, *} T_{\xi'}, \iota_{n, *} T_{\xi})$,
where
\begin{eqnarray*}
\xi \! \!  &=& \! \!\! \! \begin{cases}  {\rm pr}_1^*\log|u(\Phi_f)| & \! \!\! \! \text{if } \epsilon = 1 \\ {\rm pr}_1^*(\log |u(\Phi_{1, f})|) {\rm pr}_2^*(\pi_1(d \log u(\Phi_{2, f}))) - {\rm pr}_2^*(\log |u(\Phi_{2, f})|) {\rm pr}_1^* (\pi_1(d \log u(\Phi_{1, f}))) & \! \!\! \! \text{if } \epsilon = 2 \end{cases}, \\ 
\xi' \! \!  &=& \! \!\! \!  \begin{cases}  {\rm pr}_1^* d\log u(\Phi_f) & \text{if } \epsilon = 1 \\ {\rm pr}_1^*(d \log u(\Phi_{1, f})) \wedge {\rm pr}_2^*(d \log u(\Phi_{2, f})) & \text{if } \epsilon = 2 \end{cases}.
\end{eqnarray*}
\end{lemma}

\begin{proof}
Let us first treat the case $\epsilon = 1$. According to \cite[\S 3.7]{Jannsen}, the regulator maps are morphisms between twisted Poincar\'e duality theories. As a consequence, we have the commutative diagram
\[
\small{\begin{CD}
H^1_{\mathcal{M}}(\Sh_{\GL_2}(V_1), \Q(1)) @>p^*_{n, \mathcal{M}}>> H^1_{\mathcal{M}}(\Sh_{\H}(V), \Q(1)) @>\iota_{n, \mathcal{M},*}>> H^{d+1}_{\mathcal{M}}(\Sh_{\G}(U), \Q(t)) \\
@Vr_{\mathcal{D}}VV                @Vr_{\mathcal{D}}VV             @Vr_{\mathcal{D}}VV\\
H^1_{\mathcal{D}}(\Sh_{\GL_2}(V_1)/\R, \R(1)) @>p^*_{n,\mathcal{D}}>> H^1_{\mathcal{D}}(\Sh_{\H}(V)/\R, \R(1)) @>\iota_{n, \mathcal{D}, *}>> H^{d + 1}_{\mathcal{D}}(\Sh_{\G}(U)/\R, \R(t)).\\
\end{CD}}
\]
Via the isomorphism of Proposition \ref{DBcoh-diff-form}, the morphism $p_{n,\mathcal{D}}^*$ is induced by the pullback of differential forms. Moreover the function ${\rm pr}_1^*\log|u(\Phi_f)|$ has logarithmic singularities along the divisor at infinity. Hence the statement of the Lemma follows from Remark \ref{unitDeligne}, Proposition \ref{compatibility}, Proposition \ref{push} and Remark \ref{functorial}. The case $\epsilon = 2$ follows similarly by writing down the analogous diagram and using Proposition \ref{products}. 
\end{proof}

\subsection{A pairing on Deligne--Beilinson cohomology} In this section, we describe explicitly the Deligne--Beilinson cohomology groups in the degrees we are interested in and define a pairing between Deligne--Beilinson cohomology classes and a given harmonic differential cuspidal form. \\

Let $X = \Sh_{\G}(U)$ and $\overline{X} = \overline{\Sh_{\G}(U)}$ be a smooth toroidal compactification of $\Sh_{\G}(U)$ such that $D=\overline{\Sh_{\G}(U)}-\Sh_{\G}(U)$ is a simple normal crossing divisor.

\begin{lemma} \label{pairingDBcoh}
Let $\omega$ a smooth closed rapidly decreasing differential form on $X$ of degree $d$. Assume that $\omega$ is of type $(\tfrac{d}{2}, \tfrac{d}{2})$ if $\varepsilon=1$ and that it has components of type $(\tfrac{d-1}{2}, \tfrac{d+1}{2}), (\tfrac{d+1}{2}, \tfrac{d-1}{2})$ if $\varepsilon=2$. Then the map $\mathscr{D}^{d}({\overline{X}}) \rightarrow \C, T \mapsto T(\omega)$ induces a map 
$$
\langle \,\,, \omega \rangle: H^{d + 1}_{\mathcal{D}}(\Sh_{\G}(U), \R(t)) \rightarrow \C.
$$
\end{lemma}

\begin{proof}
By Proposition \ref{DBcoho-currents} we have
\[ H^{d+1}_{\mathcal{D}}(X, \R(t)) = \{ (S, T) \in F^t \mathscr{D}^{d+1}(\overline{X}) \oplus \mathscr{D}^d_{\R(t-1)}(\overline{X}) \} / \sim. \]
In order to show that the linear form $(S, T) \mapsto T(\omega)$ is well defined at the level of cohomology, we need to see that it vanishes at any coboundary. Let $(\tilde{S}, \tilde{T}) \in F^t \mathscr{D}^{d}(\overline{X}) \oplus \mathscr{D}^{d-1}_{\R(t-1)}(\overline{X})$. We have $d(\tilde{S}, \tilde{T}) = (d\tilde{S}, d \tilde{T} - \pi_{t-1}(\tilde{S}))$ and we need to check that $(d \tilde{T} - \pi_{t-1}(\tilde{S})) (\omega) = 0$. We have
\[ d \tilde{T}(\omega) = -\tilde{T}(d \omega) = 0 \]
since $\omega$ is closed. Moreover, $\tilde{S} \in F^t \mathscr{D}^{d}(\overline{X})$, where we recall that $t = \frac{1}{2}(d + 1 + \epsilon)$, which implies that $\tilde{S}$ vanishes on forms of type $(p, q)$ with $p < t$. If $\epsilon = 1$, by assumption $\omega$ is of type $(d/2, d/2)$ and $d / 2 < t = d/2 + 1$. If $\epsilon = 2$, $\omega$ has types $(\frac{d-1}{2}, \frac{d+1}{2}), (\tfrac{d+1}{2}, \tfrac{d-1}{2})$ and $\frac{d+1}{2} < t = \frac{d + 1}{2} + 1$. We deduce that $\tilde{S}(\omega) = 0$, which implies that the map is well defined on Deligne--Beilinson cohomology.
\end{proof}

\subsection{Integral expression for the pairing}\label{subsecintegralexpression}

In this section we state and prove the first main result of this article, which relates Beilinson's regulator of the motivic cohomology classes to an adelic integral.\\

We let $\pi=\pi_\infty \otimes \pi_f$ be a cuspidal automorphic representation of $\G(\A)$ with trivial central character such that $\pi_f^U \ne 0$. Writing $\pi_\infty|_{\G_0(\R)} \simeq \pi_\infty^1 \oplus \overline{\pi}_\infty^1$, we assume that \[\Hom_{K_\G}\left(\bigwedge^{p} \mathfrak{p}^+_{\C} \otimes_\C \bigwedge^{q} \mathfrak{p}^-_{\C}, \pi_\infty^1 \right) \neq 0\] for   $(p,q) = (d/2,d/2)$ (resp. $(p,q) = (\tfrac{d+1}{2},\tfrac{d-1}{2})$ or $(\tfrac{d-1}{2},\tfrac{d+1}{2})$  ) if $\epsilon = 1$ (resp. $\epsilon=2$). 
Furthermore, following Lemma \ref{testvector}, we consider a cusp form $\Psi = \Psi_\infty \otimes \Psi_f$ in the space of $\pi$ such that $\Psi_\infty$ is a highest weight vector of the minimal $K_\infty$-type $\tau_\infty^1$ of $\pi_\infty^1$ and such that $\Psi_f$ is a non-zero vector in $\pi_f^U$ and we let $\omega_\Psi$ be the associated harmonic cuspidal differential form. Analogously we can consider $\overline{\pi}^1_\infty$ at the place of $\pi^1_\infty$.

\begin{theorem} \label{integralform}
Let $ \overline{\Phi}:= \Phi_f \in  \mathcal{S}_0(\A_f, \overline{\Q})^{V_1}$, resp. $\overline{\Phi}:= \Phi_f \otimes \Phi'_f \in  \mathcal{S}_0(\A_f, \overline{\Q})^{V_1} \otimes_{\overline{\Q}}  \mathcal{S}_0(\A_f, \overline{\Q})^{V_2}$, if $\epsilon = 1$, resp. $\epsilon = 2$, and let $(\xi', \xi)$ be as in Lemma \ref{regulator-currents}. Then, we have
\[ \langle r_\mathcal{D}({\rm Eis}_{\mathcal{M}, n}(\overline{\Phi})), \omega_\Psi \rangle  = \int_{\Sh_\H(V)} \xi \wedge \iota_n^*\omega_\Psi. \]
\end{theorem}

\begin{proof}
According to Lemma \ref{regulator-currents}, the Deligne--Beilinson cohomology class $r_\mathcal{D}({\rm Eis}_{\mathcal{M}, n}(\overline{\Phi}))$ is represented by the pair of tempered current $(\iota_{n, *} T_{\xi'}, \iota_{n, *} T_{\xi})$ under the isomorphism of Proposition \ref{DBcoho-currents}. The result follows since from Lemma \ref{pairingDBcoh}.
\end{proof}

\subsection{The adelic integrals}\label{adelicintegralsec}

We finish the chapter using Kronecker limit formula to rewrite Theorem \ref{integralform} in terms of values of adelic Eisenstein series. Throughout the section, we let $\pi$ be a cuspidal automorphic representation as in \S \ref{subsecintegralexpression}.\\

Fix the choice of a measure on $\H(\A)$ as follows. For each finite place $p$ of $\Q$, we take the Haar measure $d h_p$ on $\H(\qp)$ that assigns volume one to $\H(\zp)$. For the archimedean place, we fix a generator $X_0$ of the highest exterior power of $\mathfrak{h}_{\C}/\mathfrak{k}_{\H,\C}$, which induces an equivalence between top differential $\omega$ forms on $X_\H = \H(\R) / K_{\H, \infty}$ and invariant measures $d_\omega h_\infty$ on $\H(\R)$ assigning measure one to $K_{\H, \infty}$ (cf. \cite[p. 83]{Harris97} for details). We then define $d h = d_\omega h_\infty \prod_{p} dh_p$.

\subsubsection{The case $n \equiv 0,3 \pmod{4}$} Let $A_n \in \mathcal{U}(\mathfrak{h})$ be the operator such that $\mathrm{pr}_{\tau_\infty^1}(X_0) = A_n. X_\infty^1$, where $\mathrm{pr}_{\tau_\infty^1}: \bigwedge^{d/2} \mathfrak{p}^+_{\C} \otimes_\C \bigwedge^{d/2} \mathfrak{p}_\C^- \rightarrow \tau_\infty^1$ denotes the projector to the $K_\infty$-type defined in Lemma \ref{testvector}.

After applying Theorem \ref{KLF} and Lemma \ref{regulator-currents}, Theorem \ref{integralform} now reads as follows.

\begin{theorem}\label{adelicintegral}
Let $n$ be congruent to $0$ or $3$ mod $4$. Let $\Phi_f \in \Sc_0(\Af^2,\overline{\Q})^{V_1}$ and let $\omega_\Psi$ be as in Theorem \ref{integralform}. We have
\begin{equation} \label{remarkonintegralform}
\langle r_\mathcal{D}({\rm Eis}_{\mathcal{M}, n}(\Phi_f)), \omega_\Psi \rangle =  \frac{h_{V}}{\mathrm{vol}(V)} \int_{\H(\Q) \Z_\G(\A) \backslash \H(\A)} E(h_1, \Phi,0) (A_n.\Psi)(h) dh,
\end{equation}
where $h_{V} = 2^{1-k(n)}| \Z_\G(\Q) \backslash \Z_\G(\Af) / (\Z_\G(\Af) \cap V)|$.
\end{theorem}

\begin{proof}
Recall that Theorem \ref{integralform} gives \[ \langle r_\mathcal{D}({\rm Eis}_{\mathcal{M}, n}(\Phi_f)), \omega_\Psi \rangle =  \int_{\Sh_\H(V)} \xi \wedge \iota_n^* \omega_\Psi, \] where, by Lemma \ref{regulator-currents} and Proposition \ref{KLF}, $\xi = {\rm pr}_1^*\log|u(\Phi_f)| = {\rm pr}_1^* E(g, \Phi, 0)$. We obtain the desired formula by passing from integrating over $\Sh_{\H}(V)$ to integrating over $\H(\Q) \Z_\G(\A) \backslash \H(\A)$ using the equivalence between top differential forms on $X_\H$ and invariant measures on $\H(\R)$ explained above. More precisely, we have
\begin{eqnarray*}
\int_{\Sh_\H(V)} \xi \wedge \iota_n^*\omega_\Psi &=& \int_{\H(\Q) \backslash \H(\A) / \Z_\H(\R) K_{\H, \infty} V} E(h_1, \Phi, 0) \omega_\Psi(X_0)(h) dh \\
&=& h_V \int_{\H(\Q) \Z_\G(\A) \backslash \H(\A) / V} E(h_1, \Phi, 0) \omega_\Psi(X_0)(h) dh \\
&=& \frac{h_V}{{\rm vol}(V)} \int_{\H(\Q) \Z_\G(\A) \backslash \H(\A)} E(h_1, \Phi, 0) \omega_\Psi(X_0)(h) dh,
\end{eqnarray*}
where we have used that $|\Z_\H(\R) /(\Z_\G \cap \H)(\R)|=2^{k(n)-1},$ with $k(n)$ being the number of elements in the partition of $n$ which defines $\H$.
Finally, note that $\omega_\Psi(X_0)(h) = (A_n.\Psi)(h)$ by definition of $\omega_\Psi$. This gives the result.
\end{proof}

\subsubsection{The case $n \equiv 1,2 \pmod{4}$} We will need the computation of the differential forms $d\log u(\Phi_f)$ in $(\mathfrak{gl}_2, K_{\GL_2})$-cohomology. To this end, recall the Cartan decomposition $\mathfrak{gl}_2=\mathfrak{k}_1 \oplus \mathfrak{p}_1^+ \oplus \mathfrak{p}_1^-$ where
$$
\mathfrak{p}_1^\pm = \left\{ \begin{pmatrix}
z & \pm iz\\
\pm iz & -z\\
\end{pmatrix} \in \mathfrak{gl}_2 \,|\, z \in \mathbb{C} \right\}.
$$
Let $v^\pm \in \mathfrak{p}_1^\pm$ denote the vector $$v^\pm=\frac{1}{2}\begin{pmatrix}
1 & \pm i\\
\pm i & -1\\
\end{pmatrix}.
$$

\begin{lemma} \label{Corodlog}
The differential form 
$$
d\log u(\Phi_f) \in \mathrm{Hom}_{K_{\GL_2}}\left( \mathfrak{gl}_2/\mathfrak{k}_1, \mathcal{C}^\infty(\GL_2(\Q) \backslash \GL_2(\A)\right))
$$
is defined by 
\begin{eqnarray*}
d\log u(\Phi_f)(v^-) &=& 0 \\
d\log u(\Phi_f)(v^+) &=& E(g,\Phi', 0),
\end{eqnarray*}
where $\Phi' = \Phi'_\infty \otimes \Phi_f$ for $\Phi'_\infty(x, y) = - 4 \pi (x^2 - y^2) e^{-\pi (x^2 + y^2)}.$ In particular,
\[ \pi_1(d\log u(\Phi_f)(v^+)) = \frac{E(g,\Phi', 0)}{2}. \]
\[ \pi_1(d\log u(\Phi_f)(v^-)) =  - \frac{\overline{E(g,\Phi', 0)}}{2}. \]
\end{lemma}

\begin{proof} Let $d=\partial+\overline{\partial}$ denote the decomposition of the differential into holomorphic and antiholomorphic part.
Since $u(\Phi_f)$ is holomorphic, it is annihilated by $\overline{\partial}$ and since $\overline{u(\Phi_f)}$ is antiholomorphic, it is annihilated by $\partial$. As a consequence
$$
d\log u(\Phi_f) = 2\partial \log |u(\Phi_f)|.
$$ 
According to the definition of the holomorphic derivative in $(\mathfrak{gl}_2, K_{\GL_2})$-cohomology (\cite[\S II.4.2 Equation (5)]{BorelWallach}), we have 
$d\log u(\Phi_f)(v^-) = 0$ and
$$
d\log u(\Phi_f)(v^+) = R_{v^+}\log u(\Phi_f)=\frac{d}{dt}(\exp(tv^+)\log u(\Phi_f))|_{t=0}.
$$

The result then follows from an easy elementary computation.
\end{proof}

Note that in the case $n \equiv 1,2 \pmod{4}$ we have $2\dim \Sh_{\H}=d+1$. Hence, we can write the generator $X_0$ of the highest exterior power of $\mathfrak{h}_{\C}/\mathfrak{k}_{\H, \C}$ as $X_0=x_1 \wedge \ldots \wedge x_{d+1}$, where $x_1$ and $x_2$ (resp. $x_3$ and $x_4$) are the images of $v^+$ and $v^-$ via the natural inclusion of $\GL_2$ in $\H$ as its first (resp. second) component. For any $\sigma \in \mathcal{S}_{d+1}$ let $X_0^\sigma=x_{\sigma(1)} \wedge \ldots \wedge x_{\sigma(d+1)}$ and let $A_{\sigma} \in \mathcal{U}(\mathfrak{g}_{\C})$ be the operator such that $\mathrm{pr}_{\tau^1_\infty}(x_{\sigma(2)} \wedge \ldots \wedge x_{\sigma(d+1)}) = A_\sigma.X_\infty^1$. Let $\Sigma_1$ and $\Sigma'_1$ (resp. $\Sigma_2$ and $\Sigma'_2$) be the subsets of permutations $\sigma$ in $\mathcal{S}_{d+1}$ such that $\sigma(1) = 1$ and $\sigma(1)=2$ (resp. $\sigma(2) = 2$ and $\sigma(2)=3$).

\begin{theorem}\label{adelicintegral2}
Let $n$ be congruent to $1$ or $2$ mod $4$. Let $\Phi_{1,f} \in \Sc_0(\Af^2,\overline{\Q})^{V_1}$, $\Phi_{2, f} \in \Sc_0(\Af^2,\overline{\Q})^{V_2}$ and let $\omega_\Psi$ be as in Theorem \ref{integralform}. We have
\begin{equation} \label{remarkonintegralform2}
\langle r_\mathcal{D}({\rm Eis}_{\mathcal{M}, n}(\Phi_{1,f} \otimes \Phi_{2, f})), \omega_\Psi \rangle = \end{equation}
\[ \frac{h_{V}}{2\mathrm{vol}(V)} \bigg( \sum_{\sigma \in \Sigma_1} (-1)^{\rm{sgn}(\sigma)} \int_{\H(\Q) \Z_\G(\A) \backslash \H(\A)} E(h_1, \Phi_1,0) E(h_2, \Phi_2,0) (A_\sigma.\Psi)(h) dh \]
\[ - \sum_{\sigma \in \Sigma'_1} (-1)^{\rm{sgn}(\sigma)} \int_{\H(\Q) \Z_\G(\A) \backslash \H(\A)} E(h_1, \Phi_1,0) \overline{E(h_2, \Phi'_2,0)} (A_\sigma.\Psi)(h) dh \]
\[ - \sum_{\sigma \in \Sigma_2} (-1)^{\rm{sgn}(\sigma)} \int_{\H(\Q) \Z_\G(\A) \backslash \H(\A)} E(h_1, \Phi'_1,0) E(h_2, \Phi_2,0) (A_\sigma.\Psi)(h) dh \]
\[ + \sum_{\sigma \in \Sigma_2'} (-1)^{\rm{sgn}(\sigma)} \int_{\H(\Q) \Z_\G(\A) \backslash \H(\A)} \overline{E(h_1, \Phi'_1,0)} E(h_2, \Phi_2,0) (A_\sigma.\Psi)(h) dh \bigg) \]
where $h_{V} = 2^{1-k(n)}| \Z_{\G}(\Q) \backslash \Z_{\G}(\Af) / (\Z_{\G}(\Af) \cap V)|$.
\end{theorem}

\begin{remark}
Observe that many of the terms vanish. Indeed, $(A_\sigma.\Psi)(h) = 0$ whenever $\widehat{x}_{\sigma(1)}$ does not belong to $\bigwedge^{\frac{d - 1}{2}} \mathfrak{p}_{\C}^+ \otimes \bigwedge^{\frac{d + 1}{2}} \mathfrak{p}_{\C}^-$, since the projection to the minimal $K_\infty$-type $\tau_\infty^1$ will be zero. The detailed computation of this integral for the case $n = 2$ has been carried out in \cite{lemmarf}.
\end{remark}

\begin{proof} According to Theorem \ref{integralform} and exactly as in the proof of Theorem \ref{adelicintegral}, we have
\[ \langle r_\mathcal{D}({\rm Eis}_{\mathcal{M}, n}(\overline{\Phi})), \omega_\Psi \rangle = \frac{h_V}{{\rm vol}(V)} \int_{\H(\Q) \Z_\G(\A) \backslash \H(\A)} (\xi \wedge  \omega_\Psi)(X_0)(h) dh, \]
By definition of the exterior product we have
$$
(\xi \wedge  \omega_\Psi)(X_0)(h) = \sum_{\sigma \in \mathcal{S}_{d + 1}} (-1)^\sigma \xi(x_{\sigma(1)}) \omega_\Psi(\widehat{x}_{\sigma(1)})(h),
$$
where $\widehat{x}_{\sigma(1)} := x_{\sigma(2)} \wedge \hdots \wedge x_{\sigma_(d + 1)}$. Observe first that, by definition, we have $\omega_\Psi(\widehat{x}_{\sigma(1)})(h) = \omega_\Psi(A_\sigma . X_\infty^1)(h) = (A_\sigma.\Psi)(h)$. By Lemma \ref{regulator-currents} we have 
$$
\xi = {\rm pr}_1^*(\log |u(\Phi_{1, f})|) {\rm pr}_2^*(\pi_1(d \log u(\Phi_{2, f}))) - {\rm pr}_2^*(\log |u(\Phi_{2, f})|) {\rm pr}_1^* (\pi_1(d \log u(\Phi_{1,f}))).
$$
According to Proposition \ref{KLF} and Lemma \ref{Corodlog}, we have
$$
{\rm pr}_1^*(\log |u(\Phi_{1, f})|)(h) = E(h_1, \Phi_{1}, 0)
$$
and
$$
{\rm pr}_2^*(\pi_1(d \log u(\Phi_{2, f})))(x_{\sigma(1)})(h) =
\begin{cases} 
0 & \! \! \text{if } \sigma(1) \neq 1, 2 \\
\frac{1}{2} E(h_2, \Phi'_{2}, 0) & \! \! \text{if } \sigma(1) = 1 \\
\frac{1}{2} \overline{E(h_2, \Phi'_{2}, 0)} & \! \! \text{if } \sigma(1) = 2, \end{cases}
$$
and a similar formula for the values ${\rm pr}_2^*(\log |u(\Phi_{2, f})|) {\rm pr}_1^* (\pi_1(d \log u(\Phi_{1,f})))(x_{\sigma(1)})$. The statement follows by putting all these calculations together.
\end{proof}

\section{Unfolding}

In this section we carry over the unfolding of the adelic integral for $\GSp_8$, $\GSp_{10}$, and  $\GSp_{14}$ that appeared above. The techniques are standard but there is a priori no general method to see whether a Rankin-Selberg integral unfolds or not (though there are some conjectures on this problem, cf. \cite{ginzburgtowards}).

\subsection{Unipotent orbits and Fourier coefficients}

Let $\mathcal{O}$ be a unipotent orbit of $\GSp_{2n}$. It corresponds to a partition $(n_1 \, n_2 \; \cdots \, n_k)$ of $2n$ with odd numbers appearing with even multiplicity. To such a $\mathcal{O}$ one attaches a set of Fourier coefficients as follows. Denote by $h_\mathcal{O}$ the one-dimensional torus \[t \mapsto \operatorname{diag}(t^{n_1-1},\cdots,t^{n_k-1},t^{n_1-3},\cdots, \cdots t^{3-n_1}, \cdots , t^{1-n_k}, \cdots,  t^{1-n_1}) \] attached to $\mathcal{O}$ (cf. \cite{nilpotentorbits}).
Given any positive root $\alpha$ (for the action of the diagonal torus of $\GSp_{2n}$), there is a non-negative integer $m$ such that \[ h_\mathcal{O}(t) x_\alpha(u) h_\mathcal{O}(t)^{-1}=x_\alpha(t^m u),\]
where $x_\alpha$ denotes the one-parameter subgroup associated to $\alpha$.
Let $U_r(\mathcal{O})$ denote the subgroup of the unipotent radical $U_\B$ of the standard (upper triangular) Borel $\B$ of $\GSp_{2n}$ generated by the $x_\alpha$ such that $m \geq r$.
Define $M(\mathcal{O}):=\T \cdot \langle x_{\pm \alpha} \; : \; m_\alpha=0 \rangle,$ where $\T$ denotes the maximal torus in $\B$. The group $M(\mathcal{O})$ acts on the characters of the abelian group $U_2(\mathcal{O}) / [U_2(\mathcal{O}),U_2(\mathcal{O})]$, and over an algebraic closure, it acts on  $U_2(\mathcal{O}) / [U_2(\mathcal{O}),U_2(\mathcal{O})]$ with an open orbit $u_2(\mathcal{O})$. Denote by ${\rm S}^0(u_2(\mathcal{O}))$ the connected component at the identity of its stabiliser. The group $M(\mathcal{O})(\Q)$ acts on the group of all additive characters of $U_2(\mathcal{O}) / [U_2(\mathcal{O}),U_2(\mathcal{O})]$, with points in $\Q \backslash \A$. Denote by $\chi_\mathcal{O}$ any non-trivial such character with stabiliser in $M(\mathcal{O})(\Q)$ of type ${\rm S}^0(u_2(\mathcal{O}))(\Q)$.
We extend $\chi_\mathcal{O}$ trivially to \[\chi_\mathcal{O}: U_2(\mathcal{O})(\Q) \backslash U_2(\mathcal{O})(\A) \to \C.\] 
Notice that such a choice is not unique and there might be infinitely many such characters which are not conjugate under the action of $M(\mathcal{O})(\Q)$ (for instance see Remark \ref{FCexGSp6}). Let $\pi$ be a cuspidal automorphic representation of $\GSp_{2n}$. If $\Psi$ is a cusp form for $\pi$, one can give the following definition.

\begin{definition}\label{fouriercoeff}
Let $\Psi$ be a cuspidal form in the space of $\pi$. Define the Fourier coefficient \begin{equation}\label{eq:firsttypeFC}
\Psi_{U_2(\mathcal{O}),\chi_{\mathcal{O}}}(g):= \int_{U_2(\mathcal{O})(\Q) \backslash U_2(\mathcal{O})(\A)} \chi_\mathcal{O}^{-1}(u) \Psi(ug)du.\end{equation}  
\end{definition}

If $U_1(\mathcal{O})/[U_2(\mathcal{O}),U_2(\mathcal{O})]$ is a generalised Heisenberg group, then it can be written as \[ X \oplus Y \oplus  U_2(\mathcal{O})/[U_2(\mathcal{O}),U_2(\mathcal{O})],\] 
where $X,Y$ are maximal abelian subgroups of $U_1(\mathcal{O})/U_2(\mathcal{O})$ which preserve $\chi_\mathcal{O}$. 
We can then define the following.

\begin{definition}\label{fouriercoeffwithextraint}
Let $\Psi$ be as above. Define the Fourier coefficient \begin{equation}\label{eq:secondtypeFC} \Psi_{U_1(\mathcal{O}),\chi_{\mathcal{O}}}(g):= \int_{X(\Q) \backslash X(\A)} \int_{U_2(\mathcal{O})(\Q) \backslash U_2(\mathcal{O})(\A)} \chi_\mathcal{O}^{-1}(u) \Psi(u x g)du dx. \end{equation}
\end{definition}

By \cite[Lemma 1.1]{GinzburgRallisSoudry}, the integral \eqref{eq:firsttypeFC} is zero for all choice of data if and only if the integral \eqref{eq:secondtypeFC} is also zero for all choice of data. We say that $\pi$ has a non-zero Fourier coefficient with respect to $\mathcal{O}$ if there is a cusp form $\Psi$ in the space of $\pi$ such that either \eqref{eq:firsttypeFC} or  \eqref{eq:secondtypeFC} is not identically zero. When no confusion arises, we will denote $\Psi_{U_2(\mathcal{O}),\chi_{\mathcal{O}}}$  (resp. $\Psi_{U_1(\mathcal{O}),\chi_{\mathcal{O}}}$) simply by $\Psi_{2,\chi}$ (resp. $\Psi_{1,\chi}$).

\begin{remark}
The representations with non-zero Fourier coefficient corresponding to the unipotent orbit $(2n)$ are precisely the generic representations.
\end{remark}

\begin{remark}\label{FCexGSp6}
The Fourier coefficients considered in \cite{PollackShah},  \cite{CLRGSp6} and \cite{CLRG2} correspond to the unipotent orbit $\mathcal{O}=(4\,2)$. In this case, \[ U_1(\mathcal{O}) = U_2(\mathcal{O}) =   \left(\begin{smallmatrix}
1 &   \star & \star & \star & \star & \star   \\ 
  & 1 &   & \star & \star & \star \\
  & & 1 & \star & \star & \star \\ 
  & & & 1 &   & \star  \\
  & & & & 1  &\star  \\
    & & & & & 1
\end{smallmatrix} \right) \] 
is the unipotent radical of the standard parabolic subgroup given by the intersection of the Siegel and Klingen parabolics. A character $\chi_\mathcal{O}:U_2(\mathcal{O})(\Q) \backslash U_2(\mathcal{O})(\A) \to \C$ can be defined by sending $u = (u_{ij}) \in U_2(\mathcal{O})(\A) / [U_2(\mathcal{O}),U_2(\mathcal{O})](\A)$ to $\psi(u_{12} + \epsilon_1 u_{24} + \epsilon_2 u_{35})$, with $\epsilon_1, \epsilon_2 \in \Q^\star/ (\Q^\star)^2$ and  $\psi: \Q \backslash \A \to \C^\times$ a non-trivial additive character.  
\end{remark}

\begin{remark}\label{unipotentorbitgsp8} If we consider the unipotent orbit $\mathcal{O}=(4\,2\,1^2)$ for $\GSp_8$, we have \[  U_2(\mathcal{O}) =   \left(\begin{smallmatrix}
1 & \star & \star & \star & \star & \star & \star & \star \\ 
  & 1 & & & & \star & \star & \star \\ 
  & & 1 & & & \star & \star & \star \\ 
  & & & 1 & & & &\star \\
  & & & & 1 & & &\star \\
  & & & & & 1 & & \star \\
  & & & & & & 1 &\star \\ 
  & & & & & & & 1
\end{smallmatrix} \right),\, U_1(\mathcal{O}) =   \left(\begin{smallmatrix}
1 & \star & \star & \star & \star & \star & \star & \star \\ 
  & 1 & & \star & \star& \star & \star & \star \\ 
  & & 1 & \star & \star & \star & \star & \star \\ 
  & & & 1 & & \star & \star&\star \\
  & & & & 1 & \star & \star &\star \\
  & & & & & 1 & & \star \\
  & & & & & & 1 &\star \\ 
  & & & & & & & 1
\end{smallmatrix} \right),\] 
with $X \simeq \mathrm{G}_a^2$ generated by the entries in position $(2,4)$ and $(2,5)$. In this case, a character $\chi_\mathcal{O}:U_2(\mathcal{O})(\Q) \backslash U_2(\mathcal{O})(\A) \to \C$ can be defined by sending $u = (u_{ij}) \mapsto \psi(u_{12} + u_{26} + u_{37})$. Note that in this case the stabiliser ${\rm S}^0(u_2(\mathcal{O})) \simeq \SL_2$ embedded into $\GSp_8$ via \[ g \mapsto \left( \begin{smallmatrix}
1 & & & & & &  \\
& 1 & & & & & \\
& & 1 & & & & \\
& & & g & & & \\
& & & & 1 & & \\
& & & & & 1 & \\
& & & & & & 1\\
\end{smallmatrix} \right) .\]
\end{remark}

\subsection{The case of $\GSp_8$}\label{s:unfoldinggsp8}

In what follows, we give a full detailed proof of the unfolding of the adelic integral for $\GSp_8$. We are indebted to Aaron Pollack and David Ginzburg, who kindly helped us with this calculation.

Let $\pi$ be a cuspidal automorphic representation of $\G = \GSp_8$ with trivial central character and let \[ I(\Phi, \varphi, s) = \int_{\H(\Q)Z_{\G}(\A) \backslash \H(\A) } {\rm E}(h_1,\Phi,s) \varphi(h) dh,\]
where ${\rm E}(g,\Phi,s)$ is the $\GL_2$-Eisenstein series of \eqref{eisensteinseriesgl2}, $\varphi$ is a cusp form in the space of $\pi$, and $\H = \GL_2 \boxtimes \GL_2 \boxtimes \GSp_4$.

Denote $S = U_{\mathbf{B}_{\GL_2}} \times U_{\mathbf{B}_{\GL_2}} \times (\SL_2 \cdot U_{\mathbf{K}_{\GSp_4}}) \subset \H$, where  $U_{\mathbf{K}_{\GSp_4}}$ denotes the unipotent radical of the (standard) Klingen parabolic $\mathbf{K}_{\GSp_4}$ of $\GSp_4$ and $\SL_2$ sits inside the Levi of $\mathbf{K}_{\GSp_4}$ through the embedding 
\[\iota: g \mapsto \left( \begin{smallmatrix} 1 & & \\ & g & \\ & & 1 \end{smallmatrix}\right). \]

\begin{proposition}\label{unfoldingforgsp8}
The integral $I(\Phi,\varphi,s)$ unfolds to  \[ \int_{S(\A)\Z_{\G}(\A) \backslash \H (\A)} f_s(h_1) \int_{[\SL_2]} \varphi_{1,\chi}( \iota(g) w_0 h) dg dh, \]
where $f_s(h_1):=f(h_1,\Phi,s)$ is the section defining the Eisenstein series, $\varphi_{1,\chi}$ is the Fourier coefficient of type $(4 \, 2 \, 1^2)$ of Definition \ref{fouriercoeffwithextraint}, and  \[w_0 = \left(\begin{smallmatrix}
1 & & & & & & &  \\ 
  & & 1 & & & &  &  \\ 
  & 1 & &  & & & &  \\ 
  & & & 1 & & & & \\
  & & & & 1 & & & \\
  & & & & & & 1 & \\
  & & & & & 1 & & \\  
  & & & & & & & 1
\end{smallmatrix} \right).\]
\end{proposition}

\begin{proof}
 First, we unfold the Eisenstein series to get 
\begin{eqnarray*}
I(\Phi,\varphi,s) &=& \int_{(\B_{\GL_2}\boxtimes \GL_2 \boxtimes \GSp_4)(\Q) \Z_\G(\A)  \backslash \H(\A)} f_s(h_1) \varphi(h) dh \\
&=& \int_{V_0(\Q) N_0(\A) \Z_{\G}(\A)  \backslash {\H}(\A)} f_s(h_1)
\int_{\Q \backslash \A} \varphi \left( \left(\begin{smallmatrix}
1 & & & & & & & y_1 \\ 
  & 1 & & & & & &  \\ 
  & & 1 & & & & &  \\ 
  & & & 1 & & & & \\
  & & & & 1 & & & \\
  & & & & & 1 & & \\
  & & & & & & 1 & \\  
  & & & & & & & 1
\end{smallmatrix} \right) h \right)dy_1 dh.
\end{eqnarray*}
Here $V_0 = (\T_{\GL_2}\boxtimes \GL_2 \boxtimes \GSp_4)$ and $N_0$ denotes the abelian subgroup of unipotent matrices whose only non-zero entry is at the upper right corner.

Let now $U_0$ denote the unipotent radical of the Klingen parabolic.  Fourier expanding over its $(\Q \backslash \A)$-points, one gets
\[\int_{V_0(\Q) N_0(\A) \Z_{\G}(\A) \backslash {\H}(\A)} f_s(h_1) \sum_{\chi \in \widehat{(\Q \backslash \A)^6}} \int_{(\Q \backslash \A)^7} \varphi \left( \left(\begin{smallmatrix}
1 & x_1 & x_2 & x_3 & x_4 & x_5 & x_6 & y_1 \\ 
  & 1 & & & & & & \star \\ 
  & & 1 & & & & &  \star \\
  & & & 1 & & & & \star \\
  & & & & 1 & & & \star \\
  & & & & & 1 & & \star \\
  & & & & & & 1 & \star \\
  & & & & & & & 1
\end{smallmatrix} \right) h \right) \chi(x) dx dy_1 dh. \]
Notice that $V_0$ acts over $U_0 / [U_0, U_0] \cong \G_a^6$ with four orbits, one of which is open.
Using automorphy of $\varphi$, one can use the action of $V_0(\Q)$ on the dual of $(U_0 / [U_0, U_0])(\Q \backslash \A) \cong (\Q \backslash \A)^6$ to regroup the sums and get
\begin{equation}\label{eq:decinorbits}
     I(\Phi,\varphi,s) = \sum_{i=0}^3
     I_{\gamma_{i}}(\Phi,\varphi,s),
\end{equation}
where $\gamma_i$ denote representatives of the orbits of this action and $I_{\gamma_i}(\Phi,\varphi,s)$ denotes
\[\int_{{\rm Stab}(\gamma_i)(\Q) N_0(\A) \Z_{\G}(\A)  \backslash \H(\A)} f_s(h_1) \int_{(\Q \backslash \A)^7} \varphi \left( \left(\begin{smallmatrix}
1 & x_1 & x_2 & x_3 & x_4 & x_5 & x_6 & y_1 \\ 
  & 1 & & & & & & \star \\ 
  & & 1 & & & & &  \star \\
  & & & 1 & & & & \star \\
  & & & & 1 & & & \star \\
  & & & & & 1 & & \star \\
  & & & & & & 1 & \star \\
  & & & & & & & 1
\end{smallmatrix} \right) h \right) \psi(\gamma_i \cdot x) dx dy_1 dh, \]
with $\psi: \Q \backslash \A \to \C^\times$ is an additive character with conductor equal to 1 at all finite places and such that $\psi_\infty(x) = e^{2 \pi i x}$ for $x \in \R$.   
Note that the open orbit is represented by $\gamma_0 = (1, 1, 0, 0, 0, 0)$ so that the corresponding character is $\psi(x_1 + x_2)$ (the other three are the trivial character, $\psi(x_1)$ and $\psi(x_2)$). The trivial orbit vanishes by cuspidality of $\varphi$ over the unipotent radical of the Klingen parabolic. Assume for the moment that the corresponding integrals associated to the two other orbits vanish. We then have
\begin{equation} \label{eaea1} I = \int_{V_1(\Q) N_0(\A) \Z_{\G}(\A) \backslash \H(\A)} f_s(h) \int_{(\Q \backslash \A)^7 } \varphi  \left( \left(\begin{smallmatrix}
1 & x_1 & x_2 & x_3 & x_4 & x_5 & x_6 & y_1 \\ 
  & 1 & & & & & & \star \\ 
  & & 1 & & & & &  \star \\
  & & & 1 & & & & \star \\
  & & & & 1 & & & \star \\
  & & & & & 1 & & \star \\
  & & & & & & 1 & \star \\
  & & & & & & & 1
\end{smallmatrix} \right) h \right)  \psi(x_1 + x_2) dx dy_1 dh.
\end{equation}
Here  $V_1 = \mathrm{Stab}(\gamma_0)$ is a product of Klingen parabolics. Ignoring the centre, one can write $V_1 = L_1 \cdot N_1$, with
\[ L_1 = \left(\begin{smallmatrix}
\lambda & & & & &  &  \\ 
  & \lambda & &  &  & &  \\ 
  & & \lambda    &  &  & &  \\
  & & & g & & & \\
   & & & & 1 & & \\
  & & & & & & 1 & \\  
  & & & & & & & 1
\end{smallmatrix} \right),\;\;  N_1 = \left(\begin{smallmatrix}
1 & & & & & & &  \\ 
  & 1 & & & & & \star &  \\ 
  & & 1 & \star &\star & \star & &  \\ 
  & & & 1 &  & \star & & \\
  & & &  & 1 & \star & & \\
  & & & & & 1 & & \\
  & & & & & & 1 & \\  
  & & & & & & & 1
\end{smallmatrix} \right), 
\]
where $g \in \GL_2$ satisfies that ${\rm det}(g)=\lambda$.
Consider now the unipotent subgroup
\[ U_1 = \left( \begin{smallmatrix}
1 & & & & & & & \\ 
  & 1 & & & & x_7 & \star &  \\ 
  & & 1 & & & & x_7 &  \\ 
  & & & 1 & & & & \\
  & & & & 1 & & & \\
  & & & & & 1 & & \\
  & & & & & & 1 & \\  
  & & & & & & & 1
\end{smallmatrix} \right) \]
so that $L_1$ acts on $U_1 / [U_1, U_1] \cong \mathbf{G}_a$ with two orbits. Fourier expand over $[U_1]$. We assume for the moment that the integral corresponding to the trivial orbit vanishes. Hence
\[ I = \int_{ V_2(\Q) N_2(\A) \Z_{\G}(\A) \backslash \H(\A)} f_s(h) \int_{(\Q \backslash \A)^{9} } \varphi \left( \left(\begin{smallmatrix}
1 & x_1 & x_2 & x_3 & x_4 & x_5 & x_6 & y_1 \\ 
  & 1 & & & & x_7 & y_2 & \star \\ 
  & & 1 & & & & \star & \star \\ 
  & & & 1 & & & &\star \\
  & & & & 1 & & &\star \\
  & & & & & 1 & & \star \\
  & & & & & & 1 &\star \\ 
  & & & & & & & 1
\end{smallmatrix} \right) h \right) \psi(x_1 + x_2 + x_7) dx dy dh, \]
where
\[ V_2 = \left(\begin{smallmatrix}
1 & & & & & & &  \\ 
  & 1 & & & & & &  \\ 
  & & 1 & \star &\star & \star & &  \\ 
  & & & \star & \star & \star & & \\
  & & & \star & \star & \star & & \\
  & & & & & 1 & & \\
  & & & & & & 1 & \\  
  & & & & & & & 1
\end{smallmatrix} \right), N_2 = \left(\begin{smallmatrix}
1 & & & & & & & \star  \\ 
  & 1 & & & & & \star &  \\ 
  & & 1 & & & & &  \\ 
  & & & 1 & & & & \\
  & & & & 1 & & & \\
  & & & & & 1 & & \\
  & & & & & & 1 & \\  
  & & & & & & & 1
\end{smallmatrix} \right) .
\]
Collapsing the integral along $V_2$, $I$ becomes equal to
\[\int_{ V_3(\A) \Z_{\G}(\A) \backslash \H(\A)} f_s(h) \int_{[\SL_2]} \int_{(\Q \backslash \A)^{12} } \varphi \left( \left(\begin{smallmatrix}
1 & x_1 & x_2 & x_3 & x_4 & x_5 & x_6 & y_1 \\ 
  & 1 & & & & x_7 & y_2 & \star \\ 
  & & 1 & y_3 & y_4 & y_5 & \star &  \star \\ 
  & & & 1 & & \star & & \star \\
  & & & & 1 & \star & & \star \\
  & & & & & 1 & & \star \\
  & & & & & & 1 & \star \\
  & & & & & & & 1
\end{smallmatrix} \right) \iota(g) h \right) \psi(x_1 + x_2 + x_7) dx dy dg dh, \]
where we denoted $V_3 = V_2 N_2$ and
\[ \iota(g) = \left( \begin{smallmatrix}
1 & & & & & &  \\
& 1 & & & & & \\
& & 1 & & & & \\
& & & g & & & \\
& & & & 1 & & \\
& & & & & 1 & \\
& & & & & & 1\\
\end{smallmatrix} \right) \] denotes the induced embedding of $\SL_2$ into $\G$ (seen as part of the Levi of the Klingen parabolic of the $\GSp_4$ component of $\H$ inside $\G$). Finally, we conjugate (possible by automorphy of $\varphi$) by the element of the Weyl group
\[ w_0 = \left(\begin{smallmatrix}
1 & & & & & & &  \\ 
  & & 1 & & & &  &  \\ 
  & 1 & &  & & & &  \\ 
  & & & 1 & & & & \\
  & & & & 1 & & & \\
  & & & & & & 1 & \\
  & & & & & 1 & & \\  
  & & & & & & & 1
\end{smallmatrix} \right) \]
which transforms the above integral as
\[\int_{ V_3(\A) \Z_{\G}(\A) \backslash \H(\A)} f_s(h) \int_{[\SL_2]} \int_{(\Q \backslash \A)^{12} } \varphi \left( \left(\begin{smallmatrix}
1 & x_2 & x_1 & x_3 & x_4 & x_6 & x_5 & y_1 \\ 
  & 1 & & y_3 & y_4 & x_7 & y_5 &  \star  \\ 
  & & 1 & & & y_2 & \star & \star  \\ 
  & & & 1 & & & \star & \star  \\ 
  & & & & 1 & & \star & \star  \\ 
  & & & & & 1 & &  \star \\
  & & & & & & 1 &  \star \\  
  & & & & & & & 1
\end{smallmatrix} \right) \iota(g) w_0 h \right) \psi(x_1 + x_2 + x_7) dx dy dg dh. \]
This shows that the integral unfolds to the Fourier coefficient $\varphi_{1,\chi}$ of type $(4 \, 2 \, 1^2)$ (cf. Remark \ref{unipotentorbitgsp8}).

We now verify the vanishing of the non-open orbits. The integrals associated to non-open orbits (apart from the identity orbit, which is zero) that appear in \eqref{eq:decinorbits} are associated to the characters $x \mapsto \psi(x_i)$ for $i=1,2$. We show that they both vanish. When $i=1$, we have
\[ \int_{V_1'(\Q) N_1'(\A) \Z_{\G}(\A) \backslash \H(\A)} f_s(h) \int_{(\Q \backslash \A)^8 } \varphi \left( \left(\begin{smallmatrix} 1 & x_1 & x_2 & x_3 & x_4 & x_5 & x_6 & y_1\\ 
  & 1 & & & & & y_2 & \star \\
  & & 1 & & & & &  \star \\ 
  & & & 1 & & & & \star \\
  & & & & 1 & & & \star \\
  & & & & & 1 & & \star \\
  & & & & & & 1 & \star \\
  & & & & & & & 1
\end{smallmatrix} \right) h \right) \psi(x_1) dx dy dh,\]
where  
\[ V'_1 = \left \{ \left(\begin{smallmatrix}
\lambda & & & &  \\ 
  & \lambda &  &  &    \\
  & & g & & \\
  & &  & 1  & \\
  & & & & 1
\end{smallmatrix} \right),\;\; g \in \GSp_4 \;:\; \nu(g)=\lambda \right \},\; \; N'_1 = N_2. \]
Consider the unipotent subgroup
\[ U_2 = \left(\begin{smallmatrix}
1 & & & & & & &  \\ 
  & 1 & x_7 & x_8 & x_9 & x_{10} & \star &  \\ 
  & & 1 & & & & \star &  \\ 
  & & & 1 & & &\star &  \\
  & & & & 1 & &\star &  \\
  & & & & & 1 & \star &  \\
  & & & & & & 1 & \\  
  & & & & & & & 1
\end{smallmatrix} \right).
\]
One has an action of $V'_1 = \GSp_4$ on $U_2 / [U_2, U_2] \cong \mathbf{G}_a^4$ and there are two orbits. By Fourier expanding, we get
\[ \int_{W'_1(\Q) N'_1(\A) \Z_{\G}(\A) \backslash \H(\A)} f_s(h) \int_{(\Q \backslash \A)^{12} } \varphi \left( \left(\begin{smallmatrix}
1 & x_1 & x_2 & x_3 & x_4 & x_5 & x_6 & y_1 \\ 
  & 1 & x_7 & x_8 & x_9 & x_{10 } & y_2 & \star  \\ 
  & & 1 & & & & \star & \star \\ 
  & & & 1 & & & \star & \star \\
  & & & & 1 & &\star & \star \\
  & & & & & 1 & \star & \star \\
  & & & & & & 1  & \star \\ 
  & & & & & & & 1
\end{smallmatrix} \right) h \right) \psi(x_1 + x_7) dx dy dh \]
\[ +
 \int_{V'_1(\Q) N'_1(\A) \Z_{\G}(\A) \backslash \H(\A)} f_s(h) \int_{(\Q \backslash \A)^{12} } \varphi\left( \left(\begin{smallmatrix}
1 & x_1 & x_2 & x_3 & x_4 & x_5 & x_6 & y_1 \\ 
  & 1 & x_7 & x_8 & x_9 & x_{10 } & y_2 & \star  \\ 
  & & 1 & & & & \star & \star \\ 
  & & & 1 & & & \star & \star \\
  & & & & 1 & &\star & \star \\
  & & & & & 1 & \star & \star \\
  & & & & & & 1  & \star \\ 
  & & & & & & & 1
\end{smallmatrix} \right) h \right) \psi(x_1) dx dy dh,
\]
with $W'_1$ isomorphic to a Klingen parabolic of $\GSp_4$.
Both integrals vanish by the cuspidality of $\varphi$, since they have as an inner integral the constant coefficient along the unipotent subgroup 
\begin{equation}\label{eq:univanishing1orbit}
   \left(\begin{smallmatrix}
1 &   & \star & \star & \star & \star & \star & \star \\ 
  & 1 & \star & \star & \star & \star & \star & \star \\
  & & 1 & & & & \star & \star \\ 
  & & & 1 & & & \star & \star \\
  & & & & 1 & &\star & \star \\
  & & & & & 1 & \star & \star \\
  & & & & & & 1  &   \\ 
  & & & & & & & 1
\end{smallmatrix} \right).
\end{equation} 

Let now $i = 2$. Conjugate the expression by $w_0$ to get
\[ \int_{V''_1(\Q) N''_1(\Q) N_0(\A) \Z_{\G}(\A) \backslash \H(\A)} f_s(h) \int_{(\Q \backslash \A)^7 } \varphi \left( \left(\begin{smallmatrix}
1 & x_2 & x_1 & x_3 & x_4 & x_6 & x_5 & y_1\\ 
  & 1 & & & & & & \star \\
  & & 1 & & & & &  \star \\ 
  & & & 1 & & &  & \star \\
  & & & & 1 & &  & \star \\
  & & & & & 1 & & \star \\
  & & & & & & 1 & \star \\
  & & & & & & & 1
\end{smallmatrix} \right) w_0 h \right) \psi(x_2) dx dy dh, \]
where  
\[ V''_1 =  \left(\begin{smallmatrix}
\lambda & & & & & & & \\ 
  & \lambda &  &  &  & & &   \\
  & & a & & & b & &\\
  & & & \star & \star & & &\\
  & & & \star & \star & & & \\
  & & c & & & d & &\\
  & & & & & & 1 & \\
  & & & & & & & 1
\end{smallmatrix} \right), N''_1 =  \left(\begin{smallmatrix}
1 & & & & & & & \\ 
  & 1 &  & \star & \star & & \star &   \\
  & & 1 & & &  & &\\
  & & & 1 &  & & \star &\\
  & & & & 1 & & \star & \\
  & &  & & & 1 & &\\
  & & & & & & 1 & \\
  & & & & & & & 1
\end{smallmatrix} \right). \]
Consider the unipotent subgroup
\[ U''_1 = \left(\begin{smallmatrix}
1 & & & & & & &  \\ 
  & 1 & x_7 & & & x_8 & \star &  \\ 
  & & 1 & & & & \star &  \\ 
  & & & 1 & & & &  \\
  & & & & 1 & & &  \\
  & & & & & 1 & \star &  \\
  & & & & & & 1 & \\  
  & & & & & & & 1
\end{smallmatrix} \right).
\]
As before, one has an action of $V''_1 = \GL_2 \boxtimes \GL_2 $ on $U''_1 / [U''_1, U''_1] \cong \mathbf{G}_a^2$ (the action is given by the natural action of the first copy of $\GL_2$ on $\mathbf{G}_a^2$) and there are two orbits. Hence the above integral equals
\[ \int_{V''_2(\Q) N''_2(\A) \Z_{\G}(\A) \backslash \H(\A)} f_s(h) \int_{(\Q \backslash \A)^{13} } \varphi \left( \left(\begin{smallmatrix}
1 & x_2 & x_1 & x_3 & x_4 & x_6 & x_5 & y_1 \\ 
  & 1 & x_7 & y_2 & y_3 & x_8 & y_4 & \star  \\ 
  & & 1 & & & y_5 & \star & \star \\ 
  & & & 1 & & & \star & \star \\
  & & & & 1 & &\star & \star \\
  & & & & & 1 & \star & \star \\
  & & & & & & 1  & \star \\ 
  & & & & & & & 1
\end{smallmatrix} \right) h \right) \psi(x_2 + x_7) dx dy dh \]
\[ +
 \int_{V''_1(\Q) N''_1(\A) N_0(\A) \Z_{\G}(\A) \backslash \H(\A)} f_s(h) \int_{(\Q \backslash \A)^{12} } \varphi \left( \left(\begin{smallmatrix}
1 & x_2 & x_1 & x_3 & x_4 & x_6 & x_5 & y_1 \\ 
  & 1 & x_7 & y_2 & y_3 & x_8 & y_4 & \star  \\ 
  & & 1 & & & & \star & \star \\ 
  & & & 1 & & & \star & \star \\
  & & & & 1 & &\star & \star \\
  & & & & & 1 & \star & \star \\
  & & & & & & 1  & \star \\ 
  & & & & & & & 1
\end{smallmatrix} \right) h \right) \psi(x_2) dx dy dh,
\]
where  
\[ V''_2 =  \left(\begin{smallmatrix}
\lambda & & & & & & & \\ 
  & \lambda &  &  &  & & &   \\
  & & \lambda & & &  & &\\
  & & & \star & \star & & &\\
  & & & \star & \star & & & \\
  & & & & & 1 & &\\
  & & & & & & 1 & \\
  & & & & & & & 1
\end{smallmatrix} \right), N''_2 =  \left(\begin{smallmatrix}
1 & & & & & & & \star \\ 
  & 1 &  & \star & \star & & \star &   \\
  & & 1 & & & \star  & &\\
  & & & 1 &  & & \star &\\
  & & & & 1 & & \star & \\
  & &  & & & 1 & &\\
  & & & & & & 1 & \\
  & & & & & & & 1
\end{smallmatrix} \right). \]
The second integral vanishes by cuspidality along the unipotent of Equation \eqref{eq:univanishing1orbit}. Making a last Fourier expansion using the unipotent group
\[ U_3 = \left(\begin{smallmatrix}
1 & & & & & & &  \\ 
  & 1 & & & &  & &  \\ 
  & & 1 & \star & \star & \star &  &  \\ 
  & & & 1 & & \star & &  \\
  & & & & 1 & \star & &  \\
  & & & & & 1 & &  \\
  & & & & & & 1 & \\  
  & & & & & & & 1
\end{smallmatrix} \right).
\] and using that $V_2'' \cong \GL_2$ acts on its abelianization (which identifies with $\G_a^2$ and the action is the natural one) with two orbits, one shows in a similar way as before that the first integral vanishes. 

We are left to show that the contribution of the trivial orbit with respect to the action of $L_1$ on $U_1 / [U_1, U_1] \cong \mathbf{G}_a$ is zero. In this case, the integral associated to the trivial orbit is
\[ \int_{  L_1(\Q) N_1(\A)N_0(\A) \Z_{\G}(\A) \backslash \H(\A)} f_s(h)  \int_{(\Q \backslash \A)^{12} } \varphi \left( \left(\begin{smallmatrix}
1 & x_2 & x_1 & x_3 & x_4 & x_6 & x_5 & y_1 \\ 
  & 1 & & y_3 & y_4 & x_7 & y_5 &  \star  \\ 
  & & 1 & & & y_2 & \star & \star  \\ 
  & & & 1 & & & \star & \star  \\ 
  & & & & 1 & & \star & \star  \\ 
  & & & & & 1 & &  \star \\
  & & & & & & 1 &  \star \\  
  & & & & & & & 1
\end{smallmatrix} \right)  w_0 h \right) \psi(x_1 + x_2) dx dy  dh. \]
Consider the unipotent subgroup $U_3$ as above 
and notice that $L_1$ acts on $U_3/ [ U_3,U_3]  \cong \mathbf{G}_a^2$ with two orbits. Thus, by Fourier expanding along $[U_3]$, we have 
\[  \int_{  V_1'''(\Q) U_{\mathbf{B}_{\H}}(\A) \Z_{\G}(\A) \backslash \H(\A)} f_s(h)  \int_{(\Q \backslash \A)^{13} } \varphi \left( \left(\begin{smallmatrix}
1 & x_2 & x_1 & x_3 & x_4 & x_6 & x_5 & y_1 \\ 
  & 1 & & y_3 & y_4 & x_7 & y_5 &  \star  \\ 
  & & 1 & x_8 & x_9 & y_2 & \star & \star  \\ 
  & & & 1 & y_6 &  \star & \star & \star  \\ 
  & & & & 1 & \star & \star & \star  \\ 
  & & & & & 1 & &  \star \\
  & & & & & & 1 &  \star \\  
  & & & & & & & 1
\end{smallmatrix} \right)  w_0 h \right) \psi(x_1 + x_2 + x_8) dx dy  dh  \]
\[+ \int_{  L_1(\Q) N_1(\A)N_0(\A) \Z_{\G}(\A) \backslash \H(\A)} f_s(h)  \int_{(\Q \backslash \A)^{12} } \varphi \left( \left(\begin{smallmatrix}
1 & x_2 & x_1 & x_3 & x_4 & x_6 & x_5 & y_1 \\ 
  & 1 & & y_3 & y_4 & x_7 & y_5 &  \star  \\ 
  & & 1 & x_8 & x_9 & y_2 & \star & \star  \\ 
  & & & 1 & &  \star & \star & \star  \\ 
  & & & & 1 & \star & \star & \star  \\ 
  & & & & & 1 & &  \star \\
  & & & & & & 1 &  \star \\  
  & & & & & & & 1
\end{smallmatrix} \right)  w_0 h \right) \psi(x_1 + x_2) dx dy  dh,  \]
with $V_1''' \subset L_1$ being a one-dimensional torus.  
The first integral vanishes by the cuspidality of $\varphi$ along the unipotent radical of the Siegel parabolic,
while the second vanishes by the cuspidality of $\varphi$ along 
\[  \left(\begin{smallmatrix}
1 &   &   &  \star & \star & \star & \star & \star \\ 
  & 1 &   & \star  & \star & \star & \star & \star \\
  & & 1 & \star & \star & \star & \star & \star \\ 
  & & & 1 &   & \star & \star & \star \\
  & & & & 1 & \star & \star & \star  \\
  & & & & & 1 &   &   \\
  & & & & & & 1  &   \\ 
  & & & & & & & 1
\end{smallmatrix} \right).\]
This completes the proof of the proposition.

\end{proof}
 
\subsection{The case  of $\GSp_{10}$}
The integral for $\G = \GSp_{10}$ involves two Eisenstein series for $\GL_2$ and is a natural generalisation of the integral for $\GSp_4$. Indeed, it unfolds to a $\GSp_{10}$-analogue of the Bessel model for $\GSp_4$. 

We let $\pi$ be a cuspidal automorphic representation of $\G$. Moreover, let $f_{i,s}(\star):=f(\star,\Phi_i,s) \in \operatorname{Ind}_{\B_2(\A)}^{\GL_2( \A)}(|\lambda|^s)$ as in \S \ref{sectionmodularunits}. For a cusp form $\varphi$ in the space of $\pi$ consider the integral 

\[I(\Phi_1,\Phi_2, \varphi, s) =  \int_{\H(\Q)Z_{\G}(\A) \backslash \H(\A) } {\rm E}(h_1,\Phi_1,s){\rm E}(h_2,\Phi_2,s) \varphi(h) dh,  \]
where $\H = \GL_2 \boxtimes \GL_2 \boxtimes \GSp_6$. 

\begin{proposition} \label{unfoldingforgsp10}
The integral $I(\Phi_1,\Phi_2, \varphi, s)$ unfolds to \[ \int_{R(\A) N_0(\A) \backslash {\H}(\A)} f_{1,s}( h_1)f_{2,s}(h_2) \int_{R(\Q) \Z_{\G}(\A)  \backslash R(\A)}
\varphi_{2,\chi}(v h) dv dh,\]
where $\varphi_{2,\chi}$ is the Fourier coefficient of type $(2^2\; 1^6)$, $N_0 = U_{\mathbf{B}_{\GL_2}} \times U_{\mathbf{B}_{\GL_2}} \times \{ I \}$, and \[R =  \left \{ \left(\begin{smallmatrix}
\lambda & & & &  \\
   & \lambda' &  & &    \\
  &  & g & & & \\
  & & & \lambda & \\
  & & & & \lambda'
\end{smallmatrix} \right),\;\; g \in \GSp_6 \;:\; \nu(g) =\lambda \lambda' \right \}.  \]
\end{proposition}

\begin{proof}
 
We start by unfolding the two Eisenstein series to get 
\begin{eqnarray*}
I(\Phi_1,\Phi_2,\varphi,s) &=& \int_{(\B_{\GL_2}\boxtimes \B_{\GL_2} \boxtimes \GSp_6)(\Q) \Z_\G(\A)  \backslash \H(\A)} f_{1,s}(h_1)f_{2,s}(h_2) \varphi(h) dh \\
&=& \int_{V_0(\Q) N_0(\A) \Z_{\G}(\A)  \backslash {\H}(\A)} f_{1,s}(h_1)f_{2,s}(h_2)
\int_{(\Q \backslash \A)^2} \varphi \left( \left(\begin{smallmatrix}
1 & & & &  y_1 \\ 
  & 1 & & y_2 &   \\ 
  & & I & &  \\ 
  & & & 1 &  \\
  & & & & 1 
\end{smallmatrix} \right) h \right)dy dh.
\end{eqnarray*}
Here $V_0 = (\T_{\GL_2}\boxtimes \T_{\GL_2} \boxtimes \GSp_6)$, $N_0 =U_{\mathbf{B}_{\GL_2}} \times U_{\mathbf{B}_{\GL_2}} \times \{ I \}$.

Let now $U_0 \simeq \G_a^3 $ denote the unipotent subgroup   \[ \{ u(z,y_1,y_2) = \left(\begin{smallmatrix}
1 & &  & z & y_1 \\ 
  & 1 & &  y_2 & z\\ 
  & & I & &  \\ 
  & & & 1 &  \\
  & & & & 1 
\end{smallmatrix} \right) \in \G, \text{ with } z,y_1,y_2 \in \G_a \}.\]
The group $V_0$ acts on $U_0$ and in particular its $\Q$-points act on the dual of $(U_0 / N_0)(\Q \backslash \A)$ with two orbits, one of which is open. Hence, Fourier expanding over the $(\Q \backslash \A)$-points of $U_0$ and using the action of $V_0(\Q)$,  $I(\Phi_1,\Phi_2,\varphi,s) = I_{\gamma_{0}}(\Phi_1,\Phi_2,\varphi,s)+ I_{\gamma_{1}}(\Phi_1,\Phi_2,\varphi,s)$, where the $\gamma_i$'s denote representatives of the orbits of this action and $I_{\gamma_i}(\Phi,\varphi,s)$ is as in the proof of Proposition \ref{unfoldingforgsp8}. Let us suppose for the moment that the integral $I_{\gamma_1}(\Phi,\varphi,s)$ corresponding to the closed orbit vanishes. Then, the integral becomes 

\[\int_{R(\Q) N_0(\A) \Z_{\G}(\A)  \backslash {\H}(\A)} f_{1,s}(h_1)f_{2,s}(h_2)
\int_{(\Q \backslash \A)^3} \varphi \left(u(z,y_1,y_2) h \right) \psi(z) dz dy_1 dy_2 dh, \]

where \[R =  \left \{ \left(\begin{smallmatrix}
\lambda & & & &  \\
   & \lambda' &  & &    \\
  &  & g & & & \\
  & & & \lambda & \\
  & & & & \lambda'
\end{smallmatrix} \right),\;\; g \in \GSp_6 \;:\; \nu(g) =\lambda \lambda' \right \}.  \]
Notice that the integral \[ \int_{(\Q \backslash \A)^3} \varphi \left(u(z,y_1,y_2) g \right) \psi(z) dz dy_1 dy_2  \]
is the Fourier coefficient $\varphi_{2,\chi}(g)$ of Definition \ref{fouriercoeff} associated to the unipotent orbit $(2^2\, 1^6)$.
We now collapse the sum over $[R]$ to get

\begin{eqnarray*}
I(\Phi_1,\Phi_2,\varphi,s) &=& \int_{R(\A) N_0(\A) \backslash {\H}(\A)} \int_{R(\Q) \Z_{\G}(\A)  \backslash R(\A)} f_{1,s}(v_1 h_1)f_{2,s}(v_2 h_2)
\varphi_{2,\chi}(v h) dv dh  \\
&=& \int_{R(\A) N_0(\A) \backslash {\H}(\A)}  f_{1,s}( h_1)f_{2,s}(h_2) \int_{R(\Q) \Z_{\G}(\A)  \backslash R(\A)}
\varphi_{2,\chi}(v h) dv dh,
\end{eqnarray*}
as desired.

We are left to show that \[I_{\gamma_{1}}(\Phi_1,\Phi_2,\varphi,s)=\int_{R(\Q) N_0(\A) \Z_{\G}(\A)  \backslash {\H}(\A)} f_{1,s}(h_1)f_{2,s}(h_2)
\int_{(\Q \backslash \A)^3} \varphi \left(u(z,y_1,y_2) h \right) dz dy dh \]
vanishes. Consider the unipotent subgroup of $\G$
\[ U_1 =  \left(\begin{smallmatrix}
1 & & \star  & \star & \star \\ 
  & 1 & \star &  \star & \star \\ 
  & & I & \star & \star \\ 
  & & & 1 &  \\
  & & & & 1 
\end{smallmatrix} \right) \simeq \mathbf{G}_a^{15}, \]
which is the unipotent radical of the parabolic subgroup $P_1$ of Levi $L_1 =\GL_2 \times \GSp_6$ whose flag parametrizes isotropic planes in the standard representation of $\G$. Notice that $U_1 = U_2(\mathcal{O})$, with $\mathcal{O} = (3^2\,1^4)$. The group $R(\Q) \subset L_1(\Q)$ acts on the dual of $(U_1/ U_0)(\Q\backslash \A)$ with two orbits, one of which is open. Taking the Fourier expansion over the $(\Q \backslash \A)$-points of $U_1$ and using the action of $R(\Q)$, we get   \[I_{\gamma_{1}}(\Phi_1,\Phi_2,\varphi,s) =  I_{\gamma_{0}'}(\Phi_1,\Phi_2,\varphi,s)+I_{\gamma_{1}'}(\Phi_1,\Phi_2,\varphi,s),\] where  $\gamma_1'$ denotes the representative of the open orbit. The integral $I_{\gamma_{0}'}(\Phi_1,\Phi_2,\varphi,s)$ associated to the identity orbit vanishes because of cuspidality of $\varphi$ along the unipotent radical $U_1$ of $P_1$. The integral corresponding to the open orbit equals  
\[ \int_{V_1(\Q) N_0(\A) \Z_{\G}(\A)  \backslash {\H}(\A)} f_{1,s}(h_1)f_{2,s}(h_2)
\int_{(\Q \backslash \A)^{15}} \varphi \left(\left(\begin{smallmatrix}
1 & & x_1 & \star & \star & \star &\star & \star &  \star \\ 
  & 1 & \star & x_2 &\star & \star & \star & \star & \star \\ 
  & & 1  &  &   &  & & \star&   \star \\
  & & & 1  &   &  & &   \star  & \star \\
  & & & & I_2 & &  & \star & \star \\
   & & & & & 1 & & \star& \star\\
& & & & & & 1 & \star & \star\\
    & & & & & & & 1 & \\
  & & & & & & & & 1
\end{smallmatrix} \right) h \right) \psi(x_1 + x_2) dx dh, \]
with $V_1 = M_1 N_1$, where $N_1 = \{ 1 \} \times \{ 1 \} \times U_2$ is isomorphic to the unipotent radical of the maximal standard parabolic $P_2 = L_2 U_2$ of $\GSp_6$, with $L_2 \simeq \GL_2^2$, and 
\[ M_1 =  \left \{ \left(\begin{smallmatrix}
\lambda & &   &   &  &  &  &   &   \\ 
  & \lambda' &  &  & &  & &  &  \\ 
  & & \lambda  &  &   &  & & & \\
  & & & \lambda' &  &  & &  &  \\
  & & & & g & &  &   &   \\
   & & & & & \lambda & &  &  \\
& & & & & & \lambda' &  & \\
    & & & & & & & \lambda & \\
  & & & & & & & & \lambda'
\end{smallmatrix} \right),\;\; g \in \GL_2\,:\, \nu(g)=\lambda\lambda' \right \}.\]
After collapsing the sum over $[N_1]$, the integral $I_{\gamma_{1}'}(\Phi_1,\Phi_2,\varphi,s)$ contains as an inner integral the period of $\varphi$ over the unipotent radical subgroup \[ \left(\begin{smallmatrix}
1 & & & &\star  & \star & \star &\star & \star &  \star \\ 
 & 1 & &  & \star &\star & \star & \star & \star & \star \\ 
  & & 1 &   & \star & \star  & \star &\star & \star&   \star \\
  & & & 1 & \star  & \star  & \star &\star &   \star  & \star \\
 & & & & 1&  &\star & \star & \star & \star \\
  & & & & & 1& \star & \star & \star & \star\\
  & & & & & & 1 & &  &  \\
& & & & & & & 1 &   & \\
& & & & & & & & 1 & \\
& & & & & & & & & 1
\end{smallmatrix} \right),\]
and thus $I_{\gamma_{1}'}(\Phi_1,\Phi_2,\varphi,s)$ vanishes because of cuspidality of $\varphi$.

\end{proof}

\subsection{The case of $\GSp_{14}$}

We now describe the unfolding for the integral when $\G=\GSp_{14}$. We are indebted to David Ginzburg, who crucially helped us to complete the proof to Proposition below. 

As in \S \ref{s:unfoldinggsp8}, we let $\pi$ be a cuspidal automorphic representation of $\G$ with trivial central character and consider \[ I(\Phi, \varphi, s) = \int_{\H(\Q)Z_{\G}(\A) \backslash \H(\A) } {\rm E}(h_1,\Phi,s) \varphi(h) dh,\]
where $\varphi$ is a cusp form for $\pi$  and $\H = \GL_2 \boxtimes  \GSp_4  \boxtimes \GSp_8$. We have the following. 

\begin{proposition} \label{unfoldingforgsp14}
The integral $ I(\Phi, \varphi, s)$ unfolds to an adelic integral involving a Fourier coefficient $\varphi_{1,\chi}$ associated to the unipotent orbit $(4 \, 2 \, 1^8)$.  
\end{proposition}

\begin{proof}
The proof is very similar to the one of Proposition \ref{unfoldingforgsp8}, so we limit ourselves to sketch it here. After unfolding the Eisenstein series, we get 
\begin{eqnarray*}
I(\Phi,\varphi,s) &=& \int_{V_0(\Q) N_0(\A) \Z_{\G}(\A)  \backslash {\H}(\A)} f_s(h_1)
\int_{\Q \backslash \A} \varphi \left( \left(\begin{smallmatrix}
1 &  & y_1 \\ 
  & I &  \\ 
  & & 1
\end{smallmatrix} \right) h \right)dy_1 dh,
\end{eqnarray*}
where $V_0 = (\T_{\GL_2} \boxtimes \GSp_4 \boxtimes \GSp_8)$ and $N_0$ denotes the abelian subgroup of unipotent matrices whose only non-zero entry is at the upper right corner.

Fourier expanding over the $(\Q \backslash \A)$-points of the unipotent radical $U_0$ of the Klingen parabolic, one gets
\[\int_{V_0(\Q) N_0(\A) \Z_{\G}(\A) \backslash {\H}(\A)} f_s(h_1) \sum_{\chi \in \widehat{(\Q \backslash \A)^{12}}} \int_{(\Q \backslash \A)^{13}} \varphi \left( \left(\begin{smallmatrix}
1 & x_1 & \cdots & x_{12} & y_1 \\ 
  & 1 & & & \star \\ 
  & & I & &  \star \\
  & & & 1 & \star \\
  & & & & 1
\end{smallmatrix} \right) h \right) \chi(x) dx dy_1 dh. \]
Notice that $V_0$ acts over $U_0 / [U_0, U_0]$ with four orbits, one of which is open.
As in the proof of Proposition \ref{unfoldingforgsp8}, we can use this action to write
\begin{equation*}
     I(\Phi,\varphi,s) = \sum_{i=0}^3
     I_{\gamma_{i}}(\Phi,\varphi,s),
\end{equation*}
where the $\gamma_i$'s denote representatives of the orbits of this action and $I_{\gamma_i}(\Phi,\varphi,s)$ denotes
\[\int_{{\rm Stab}(\gamma_i)(\Q) N_0(\A) \Z_{\G}(\A)  \backslash \H(\A)} f_s(h_1) \int_{(\Q \backslash \A)^{13}} \varphi \left( \left(\begin{smallmatrix}
1 & x_1 & \cdots & x_{12} & x_{13} \\ 
  & 1 & & & \star \\ 
  & & I & &  \star \\
  & & & 1 & \star \\
  & & & & 1
\end{smallmatrix} \right) h \right) \psi(\gamma_i \cdot x) dx dh, \]
with $\psi: \Q \backslash \A \to \C^\times$ is the additive character fixed in \emph{loc.cit.}. The three integrals corresponding to the closed orbits vanish by cuspidality of $\varphi$ (the proof is identical to the one for $\GSp_8)$. 
   
The open orbit is represented by $\gamma_0 = (1, 0, 1, 0,\cdots , 0, 0)$ so that the corresponding character is $\psi(x_1 + x_3)$ with stabiliser $V_1$ equal to a product of Klingen parabolics. Denote $V_1=L_1 \cdot N_1$ its Levi decomposition. 
Then, collapsing the integral over $[N_1]$ and conjugating by the Weyl element
\[ w_0 = \left(\begin{smallmatrix}
 1 & & & & & & & &  \\ 
  & & & 1 & & & &  &  \\ 
  & 1 & & & & & &  &  \\
  &  & 1 & & & & &  &  \\ 
  &  & & & I_8 & & &  & \\ & & & & & & & 1 & \\
  & & & & & 1 & & & \\
  & & & & & & 1 & & \\  
  & & & & & & & & 1
\end{smallmatrix} \right), \] 
where $I_r$ denotes the $r\times r$ identity matrix, we get that  $ I(\Phi,\varphi,s) = I_{\gamma_{i}}(\Phi,\varphi,s)$ equals \begin{equation*} \int f_s(h) \int_{(\Q \backslash \A)^{23} } \varphi  \left( \left(\begin{smallmatrix}
1 & x_3 & x_1 & x_2 & x_4 & \cdots & x_9 & x_{10} & x_{11} & x_{12}& x_{13} \\ 
  & 1 &  &  & y_1 & \cdots & y_6 & &  & y_7 & \star \\ 
  & & 1 & y_8 & & & & y_9 & y_{10} & &  \star \\
  & & & 1 & & & &  & \star&   & \star \\
  & & & & 1 & & &   & & \star & \star \\
  & & & & & I_4 & &  & & 
  \star & \star \\
  & & & & & & 1 &   & &  \star & \star \\
  & & & & & & & 1 &  &    & \star \\
  & & & & & & & & 1 &  & \star\\
  & & & & & & & & & 1 & \star\\
  & & & & & & & & & & 1
\end{smallmatrix} \right) w_0 h \right)  \psi(x_1 + x_3) dx dy dh,
\end{equation*}
where the first integral is taken over $L_1(\Q) N_1(\A)N_0(\A) \Z_{\G}(\A) \backslash \H(\A)$.

Consider now the unipotent subgroup
\[ U_1 = \left( \begin{smallmatrix}
1 & & &  & & & \\ 
  & 1 & & & x_{14} & \star &  \\ 
  & & 1 &  & & x_{14} &  \\ 
  & & & I_8  & & & \\
  & & & &  1 & & \\
  & & & & &  1 & \\  
  & & & & & &  1
\end{smallmatrix} \right) \]
so that $L_1$ acts on $U_1 / [U_1, U_1] \cong \mathbf{G}_a$ with two orbits, one open and one closed. Fourier expand over $[U_1]$. Again, one shows that the integral corresponding to the trivial orbit vanishes. 
After applying the global root exchange to move $y_8$ and $y_9$ respectively to positions $(2,11)$ and (2,4) and \cite[Lemma 7.1]{GinzburgRallisSoudryDescent}, we get that the integral equals 
\begin{equation*} \int_{\A^2} \int_{D} f_s(h) \int_{(\Q \backslash \A)^{24} } \varphi  \left( \left(\begin{smallmatrix}
1 & x_3 & x_1 & \star &\star & \star &  \star \\ 
  & 1 & & \star & x_{14} & \star & \star \\ 
  & & 1  &  &  \star & \star &  \star \\
  & & & I_8  &   &  \star  & \star \\
  & & & & 1  & & \star \\
   & & & & & 1 & \star\\
  &  & & & & & 1
\end{smallmatrix} \right) \iota(y_8,y_9) w_0 h \right)  \psi(x_1 + x_3+ x_{14}) dxdh dy_8 dy_9.
\end{equation*}
where $D=V_2(\Q)N_2(\A) \Z_{\G}(\A) \backslash \H(\A)$, with $ V_2 \simeq \SL_2 \times \Sp_6$, $N_2$ denotes the unipotent radical of the Klingen parabolic for $\H$, and $\iota(y_8,y_9)$ denotes the two dimensional unipotent subgroup generated by the entries $y_8, y_9$.

Notice that similar to the case of $\GSp_8$, the inner integral is the Fourier coefficient $\varphi_{1,\chi}$ associated to the unipotent orbit $(4 \, 2 \, 1^8)$. Finally, collapsing the integral along $V_2$, we finally get 
\begin{equation*} \int_{\A^2} \int_{(\SL_2 \times \Sp_6)(\A)N_2(\A) \Z_{\G}(\A) \backslash \H(\A)} f_s(h)  \int_{[\SL_2 \times \Sp_6]}   \varphi_{1,\chi}( r(g_1,g_2) \iota(y_8,y_9) w_0 h) dg dh dy_8 dy_9,
\end{equation*}
where $r:\SL_2 \times \Sp_6 \hookrightarrow \GSp_{14}$ is the embedding  
\[(g_1,g_2)= \left( \left( \begin{smallmatrix} a & b \\ c & d \end{smallmatrix} \right), g_2 \right) \mapsto \left(\begin{smallmatrix}
I_3 &   &   & & \\ 
  & a &   & b  & \\
  & & g_2  & &\\ 
  & c & & d & \\
  & & & & & I_3
\end{smallmatrix} \right).\]

\end{proof}

\bibliographystyle{alpha}
\bibliography{bibliography}

\end{document}